\documentclass[11pt]{amsart}
\usepackage{epsfig,psfrag,verbatim,amssymb,amsmath,amsfonts,graphicx,bbm}
\newtheorem{theorem}{Theorem}[section]
\newtheorem{lemma}[theorem]{Lemma}
\newtheorem{prop}[theorem]{Proposition}
\newtheorem{cor}[theorem]{Corollary}

\newtheorem{remark}[theorem]{Remark}
\newtheorem{rp}[theorem]{Remark and Problem}
\newtheorem{re}[theorem]{Remark and Example}
\newtheorem{defi}[theorem]{Definition}
\newtheorem{problem}[theorem]{Problem}
\newtheorem{pe}[theorem]{Problem and Example}
\numberwithin{equation}{section}

\newcommand{\EE}{\mathbb{E}}

\newcommand{\R}{\mathbb{R}}

\newcommand{\Z}{\mathbb{Z}}
\newcommand{\N}{\mathbb{N}}

\newcommand{\PP}{\mathbb{P}}

\newcommand{\wom}{\widetilde{\om}}

\renewcommand{\tilde}{\widetilde}

\newcommand{\M}{\mathcal{M}}
\newcommand{\F}{\mathcal{F}}
\newcommand{\la} {\lambda}
\newcommand{\si}{\sigma}
\newcommand{\al}{\alpha}
\newcommand{\ga}{\gamma}
\newcommand{\om}{\omega}
\newcommand{\Om}{\Omega}
\newcommand{\eps}{\varepsilon}
\renewcommand{\epsilon}{\varepsilon}
\newcommand{\1}{\mathbf{1}}
\newcommand{\won}{{\boldsymbol 1}}
\newcommand{\cal}{\mathcal}

\newcounter{constante}
\newcommand{\con}[1]{
\immediate\write 1{\noexpand\newlabel{#1}{{\theconstante}{\theconstante}}}
                    c_{\theconstante}
                    \stepcounter{constante}
                   }

\begin{document}

\setcounter{page}{1}

\title[Excited random walks] {Excited random walks:\\ results, methods, open problems}

\author{Elena Kosygina and Martin P.W.\ Zerner} 

\thanks{\textit{2000
Mathematics Subject Classification.} 60K35, 60K37,
  60J80.}
\thanks{\textit{Key words:}\quad 
excited random
walk, cookie walk, recurrence, transience, zero-one laws, law of large
numbers, limit theorems, random environment, regeneration structure.}

\begin{verse}
  \makebox[5cm]{\ }\hfill\textit{Dedicated to Professor S.R.S.\
  Varadhan}\\ \makebox[53.5mm]{\ }\textit{on the occasion of his 70-th birthday.}
\end{verse}\vspace*{10mm}

\begin{abstract}
  We consider a class of self-interacting random walks in
  deterministic or random environments, known as excited random walks
  or cookie walks, on the $d$-dimensional integer lattice. The main
  purpose of this paper is two-fold: to give a survey of known
  results and some of the methods and  to present several new
  results. The latter include functional limit theorems for transient
  one-dimensional excited random walks in bounded i.i.d.\ cookie
  environments as well as some zero-one laws. Several open problems
  are stated.
\end{abstract}
\maketitle

\pagestyle{myheadings}
\markboth{Excited random walks}{E.\ Kosygina, M.P.W.\ Zerner}

\section{Model description}

Random walks (RWs) and their scaling limits are probably the most widely
known and frequently used stochastic processes in probability theory,
mathematical physics, and applications. Studies of a RW  in a
random medium are an attempt to understand which macroscopic effects
can be seen and modeled by subjecting the RW's dynamics on a
microscopic level to various kinds of noise, for example, allowing it
to interact with a random environment or its own history (through path
restrictions, reinforcement, excitation etc.).

The Markov, or memory-less, property of simple  RWs lies at the
heart of the classical approach to these and much more general
processes. But it also imposes a restriction on the applicability of
such models, as many real life processes certainly have memory. Over
the last decades several types of non-markovian RWs appeared
in the literature and became active areas of research; for example,
self-avoiding RWs, edge or vertex reinforced RWs, and
excited RWs (ERWs), also known as ``cookie
walks''.\footnote{According to Itai Benjamini [personal
  communication], the name ``excited random walk'' for the model
  studied in the seminal paper \cite{BW03} was suggested by Oded
  Schramm. The notion of ``cookies'' in this context was introduced
  later in \cite{Zer05}.}  The reader interested in the first two
types of models is referred to the surveys \cite{T99} and \cite{P07}. At the
time when \cite{P07} was written, ERWs had just appeared and for
this reason were only briefly mentioned \cite[p.\,51]{P07}.

The main purpose of this article is two-fold. First, we give a survey of
results and some of the methods concerning ERWs. Second, we include
several new theorems, see e.g.\ Theorems \ref{R2}, \ref{kal2},
\ref{1to2}, and \ref{2up}.  We also state several open problems. While
we aim at presenting all major results known for ERWs on $\Z^d, d\ge
1,$ the choice of methods explained in some detail has been influenced
by our personal preferences and contributions to this area.

Primarily, we shall be concerned with discrete time ERWs on $\Z^d$,
$d\ge 1,$ even though the construction can be readily extended to
other graphs (see \cite{V03,DGDMP08,BS09} for trees, \cite{Zer06,
  Do11} for strips) or continuous time-space processes (see
\cite{RS11, RS} for so-called excited Brownian motion). Broadly
speaking, one considers a certain underlying and presumably
well-understood process and modifies its dynamics for the first few
visits to each site. These modifications can be thought of as
stacks of {\em cookies} placed on each site of the lattice. Each
cookie encodes a probability distribution on $\Z^d$. The walker
consumes a cookie at his current location and makes a move according
to the distribution prescribed by that cookie. Upon reaching a site
where all cookies have already been eaten or were not there to begin
with, the walker makes a move in accordance with the original underlying
dynamics.  Below we consider processes whose underlying dynamics is
the simple symmetric RW and for which the cookie stacks are
random themselves (i.i.d.\ or stationary and ergodic) and induce
transitions between nearest neighbors. (For a model with cookies
inducing (long-range) non-nearest neighbor transitions see
\cite{KRS}.)

We start with the description of a relatively general model which
allows infinite cookie stacks.  Let $\mathcal E:=\{\pm e_j\mid
j\in\{1,2,\dots,d\}\}$ be the set of unit coordinate vectors in $\Z^d$
and denote by ${\cal M}_{\mathcal E}$ the set of probability measures
on $\mathcal E$, i.e.\ vectors with $2d$ non-negative entries which sum
up to 1. Such vectors are called \textit{cookies}.
The set of cookie environments is denoted by 
\[  \Omega:=\cal M_{\mathcal E}^{ \Z^d\times \N}.\]
(Here $\N=\{1,2\ldots\}$.)  The elements of $\Omega$ are written as
$\omega=(\omega(z,e,i))_{z\in\Z^d,e\in \mathcal E,i\in\N}$ with
$(\omega(z,e,i))_{e\in \mathcal E}$ being the $i$-th cookie at $z$. It
is consumed by the walker upon the $i$-th visit to $z$, if there is
such a visit, and provides the walker with the transition
probabilities from $z$ to $z+e$ in the next step.  More precisely, for
fixed $\omega\in\Omega$ and $x\in\Z^d$ an ERW starting at $x$ in the
environment $\omega$ is a process $X:=(X_n)_{n\ge 0}$ on a suitable
probability space $(\Om',\F',P_{x,\omega})$ which satisfies
\begin{eqnarray}\label{top}
P_{x,\omega}[X_0=x]&=&1 \quad\mbox{and}\\
P_{x,\omega}\left[X_{n+1}=X_n+e \mid (X_{i})_{0\le i\le n}\right]&=&
\omega\left(X_n,e,\#\{i\in \{0,1,\dots,n\} \mid X_i=X_n\}\right)
\nonumber
\end{eqnarray}
for all $n\in\N_0=\N\cup\{0\}$ and $e\in \mathcal E$. Here $\#A$
denotes the cardinality of the set $A$.

The cookie environment $\omega$ may be chosen at random itself
according to a probability measure $\PP$ on $(\Omega,\F)$, where $\F$
is the canonical product Borel $\si$-algebra.  Averaging the so-called
{\em quenched} measure $P_{x,\om}$ over the environment $\om$ we
obtain the {\em averaged} (often also called {\em annealed}) measure
$P_x[\cdot]:=\EE[P_{x,\om}[\cdot]]$ on $\Om\times\Om'$.  The
expectation operators corresponding to $P_{x,\om}, \PP,$ and $P_x$ are
denoted by $E_{x,\om}, \EE,$ and $E_x$, respectively.

A common assumption about $\PP$ will be:
\begin{equation*}
\mbox{the family $(\omega(z,\cdot,\cdot))_{z\in\Z^d}$ of cookie stacks is i.i.d.\ under
$\PP$.}
\tag{IID}
\end{equation*}
A weaker condition is:
\begin{equation*}
\mbox{\begin{tabular}{l}the family $(\omega(z,\cdot,\cdot))_{z\in\Z^d}$ is under
$\PP$ stationary \\
 and ergodic\footnotemark\ with respect to the shifts on $\Z^d$.
\end{tabular}}\tag{SE}
\end{equation*}
\footnotetext{i.e.\ every $A\in\F$ which is invariant under all shifts on $\Z^d$ satisfies $\PP[A]\in\{0,1\}$.}

To escape degenerate situations we shall often assume one of the
following \textit{ellipticity conditions}, called weak ellipticity,
ellipticity, and uniform ellipticity.
\begin{equation*}
 \mbox{For all $z\in\Z^d, e\in\mathcal E$:\ }\  
\PP\left[\forall i\in\N:\ \om(z,e,i)>0\right]>0.\tag{WEL}
\end{equation*}
\[  \mbox{For all $z\in\Z^d, e\in\mathcal E$ and $i\in\N:\ \PP$-a.s.\ $\om(z,e,i)>0$.}\tag{EL}
\]
\begin{equation*}
  \mbox{There is $\kappa>0$ such that for all $z\in\Z^d, e\in\mathcal E$ and $i\in\N:\ \PP$-a.s.\ $\om(z,e,i)\ge \kappa$.}\tag{UEL}
\end{equation*}
Obviously, (UEL)$\Rightarrow$(EL)$\Rightarrow$(WEL).

If $\omega \in\Omega$ is such that $\omega(z,e,i)$ does not depend on
$z,\ e$, and $i$ then necessarily $\omega(z,e,i)=1/(2d)$ and hence $X$
is, under $P_{x,\omega}$, the simple symmetric RW on $\Z^d$ starting
at $x$. Since the consumption of a cookie $(\om(z,e,i))_{e\in \mathcal
  E}=(1/(2d))_{e\in \mathcal E}$ does not change the dynamics of the
simple symmetric RW, which is our underlying process, such a cookie
will be called a \textit{placebo cookie}.

If we allow dependence of $\om(z,e,i)$ on $e$ but neither
on $z$ nor $i$ we obtain all  nearest neighbor random
walks on $\Z^d$ starting at $x$  whose increments are i.i.d..

If we let $\omega(z,e,i)$ depend on $e$ and $z$ but still not on $i$
then we get all nearest neighbor Markov chains on $\Z^d$. Choosing
$\omega\in\Omega$ at random (assuming usually (IID) or (SE)) yields
all nearest neighbor \textit{RWs in random environments} (RWRE) on
$\Z^d$, see e.g.\ \cite{Zei04, S04}. (However, note that although RWRE
can be thought of as a special case of ERW, the present paper does not
attempt to survey RWRE results.)

All these processes are markovian.  However, as soon as $\om(z,e,i)$
starts depending on $i$, $X$ loses the Markov property. The walker's
behavior may depend on the number of visits to his current
location.
Such class of processes  seems to be too general to be
considered in a coherent way and might not be related anymore to any
well-understood underlying process. Thus one has to impose further
conditions on $\PP$.  

One way to re-establish the connection to the underlying process, in
our case to the simple symmetric RW, is to limit the number of
non-placebo cookies per site. To this purpose, we denote by
\begin{equation}\label{MM}
M(z):=\inf\{j\in\N_0\mid \forall e\in \mathcal E\ \forall i>j:
\omega(z,e,i)=1/(2d)\}
\end{equation}
the number of cookies at site $z\in\Z^d$, if we do not count placebo
cookies which are only followed by more placebo cookies.  (Here
$\inf\emptyset=\infty$.)  Note that under assumption (IID) (resp.\
(SE)) $M(z),\ z\in\Z^d,$ is an i.i.d.\ (resp.\ stationary and ergodic)
family of random variables. Again, if $M\equiv 0$ then $(X_n)_{n\ge
  0}$ is the simple symmetric RW, whereas for RWRE we have $M(z)\in\{0,\infty\}$.

 In the next four short subsections we present the most
commonly studied measures $\PP$ on $\Om$.

\subsection {The original ERW model  and a modification (no
  excitation after the first visit)}\label{ss11}
ERWs were introduced by Benjamini and Wilson in \cite{BW03} for $\Z^d,
d\ge 1$, with emphasis on $d\ge 2$. There the walker gets a bias in
the direction $e_1$ upon the first visit to a site, whereas upon
subsequent visits to the same site he jumps to a uniformly chosen
neighbor.  More precisely, in the notation introduced above, there is
a $p\in(1/2,1]$ such that $\PP=\delta_\om$, where
for all $z\in\Z^d$, 
\[\begin{array}{l}
  {\displaystyle \omega(z,e_1,1)=\frac{p}{d},\quad
  \omega(z,-e_1,1)=\frac{1-p}{d},\ \text{and }}\vspace*{1mm}\\
  {\displaystyle \omega(z,e,i)=\frac{1}{2d} \quad \text{if }\ i\in\N \text{ and
  }e\in{\cal E}\setminus\{e_1,-e_1\} \ \ \text{or \ \ if }\ i\ge
  2.}
\end{array}\tag{BW}
\]
In this special environment $\om$ one can compare the ERW with a
simple RW $(Y_n)_{n\ge 0}$ by coupling these processes so
that $(X_n-Y_n)\cdot e_1$ is non-decreasing in $n$ and $X_n\cdot
e_i=Y_n\cdot e_i$ for all $i=2,\ldots,d$ and $n\ge 0$.

A natural extension of this model is obtained by fixing an arbitrary
direction $\ell\in\R^d\backslash\{0\}$ and assuming that all the
cookies induce a drift in that direction.
In \cite[Th.\ 1.2]{MPRV} the setting is extended to measures $\PP$
which satisfy (IID), (UEL), and for which there is an
$\ell\in\R^d\backslash\{0\}$ such that
\[\begin{array}{l}
{\displaystyle \exists \la>0: 
\sum_{e\in{\cal E}}\omega(0,e,1)\ e\cdot \ell\ge \la\quad \mbox{$\PP$-a.s.\ and}}\\
{\displaystyle
   \omega(0,e,i)=\om(0,-e,i)\ \ \text{$\PP$-a.s.\ for all }e\in\mathcal E,\ i\ge 2.}
\end{array}
\tag{MPRV$_{\ell}$} 
\]

\subsection{Any number of $\ell$-positive
  cookies per site}\label{1.3}
One way to further generalize the models described in Section
\ref{ss11} is to lift the restriction on the number of non-placebo
cookies, still requiring each cookie to induce a non-negative drift in the
same general direction.  Such a model was introduced in \cite{Zer05}
for $\Z$ and in \cite{Zer06} for $\Z^d, d\ge 2$ (and for strips as a
graph in between $\Z$ and $\Z^2$). It is assumed that $\PP$ satisfies
(SE) if $d=1$ and (IID) and (UEL) if $d\ge 2$ and that
for some $\ell\in \R^d\backslash\{0\},$
\begin{equation*}
  \sum_{e\in\mathcal E}\omega(0,e,i)\ e\cdot\ell\ge 0\quad\mbox{ $\PP$-a.s.\ for all $i\in\N$.}\tag{POS$_\ell$}
\end{equation*}
A useful fact in this setting is that $(X_n\cdot\, \ell)_{n\ge 0}$ is
a $P_{0,\omega}$-submartingale. 

Cookies which satisfy the inequality
in (POS$_\ell$) are called ($\ell$-)\textit{positive}. (Negativity is
defined by the opposite inequality.)

\subsection{Boundedly many positive or negative
  cookies per site ($d=1$)}\label{1.4}
This model was introduced in \cite{BS08a, KZ08}. 
Assume (IID), (WEL), and 
that the number of non-placebo cookies and their positions within the cookie stacks are bounded, i.e.
\begin{equation*}
 \mbox{there is a deterministic $M\in\N$ such that for all $z\in\Z^d$: $\PP$-a.s.\ $M(z)\le M$.
}\tag{BD} 
\end{equation*}
This model is probably the most studied. Currently, there is a
rather complete picture, which includes  criteria for
recurrence and transience, laws of large numbers, ballisticity,
functional limit theorems, and large deviations.

This model can also be considered in higher dimensions, but so far
there has been little progress for $\Z^d$, $d\ge 2$.  
Practically all work that is done on this model for $d=1$ uses a
connection with branching processes with migration. This connection
allows to translate main questions for this non-markovian model into
questions about branching processes, which are markovian. The
branching process approach was also useful in considering ERW on trees
(see \cite{BS09}) but it completely breaks down on $\Z^d$,
$d\ge 2$.

\subsection{RW perturbed at extrema ($d=1$)}\label{ext}
One of the historically first studied non-markovian RWs which fits
into the above setting is the so-called {\em RW perturbed at
  extrema}, see e.g.\ \cite{D99} and the references therein. This is a
nearest-neighbor process $(Y_n)_{n\ge 0}$ on $\Z$ which starts at 0
and satisfies
\[P[Y_{n+1}=Y_n+1\mid (Y_i)_{0\le i\le n}]=\left\{\begin{array}{ll}
p&\mbox{if  $n\ge 1$ and $Y_n=\max_{i\le n}Y_i$,}\\
q&\mbox{if  $n\ge 1$ and $Y_n=\min_{i\le n}Y_i$,}\\
1/2 &\mbox{otherwise}
\end{array}\right.\]
for given $p,q\in(0,1]$.  It has been noted in \cite[Sec.\ 1]{BV} that
this walk can be viewed as an ERW under the averaged measure as
follows: Let $(M_z)_{z\in\Z}$ be independent random variables w.r.t.\
some probability measure $P$ such that $P[M_z=k]=p(1-p)^{k-1}$ for
$k,z\ge 1$, $P[M_z=k]=q(1-q)^{k-1}$ for $k\ge 1, z\le -1$ and
$M_0\equiv 0$.  Set for all $z\in\Z$ and $i\in\N$,
\[\wom(z,1,i)=\left\{\begin{array}{cl}
1&\mbox{if $z<0$ and $i<M_z$ or if $z>0$ and $i=M_z$,}\\
0&\mbox{if $z<0$ and $i=M_z$ or if $z>0$ and $i<M_z$,}\\
1/2&\mbox{if $i>M_z$}
\end{array}\right.
\]
and let $\PP$ be the distribution of $\wom$ on $\Omega$. (Note that
$(M_z)_{z\in\Z}$ and $(M(z))_{z\in\Z}$ as defined in (\ref{MM}) have
the same distribution unless $p=1/2$ or $q=1/2$.) Then the ERW $X$ has
under $P_0$ the same distribution as $(Y_n)_{n\ge 0}$. However, from
the point of view of ERW as introduced above this measure $\PP$ seems
a bit unnatural since it does not satisfy (SE). Therefore, this model
will not play an important role in this paper.
For a variant of this model which does satisfy (SE) see
\cite{P10} and Remark \ref{haveit}. 
\vspace*{5mm}

Let us describe how the present paper is
organized. At the end of this section we introduce some notation.  In
Section \ref{rara} we deal with the basic question when the range of
the ERW is finite and when it is infinite.  This will be useful in Section
\ref{srt}, which considers recurrence and transience. There we discuss
zero-one laws for (directional) recurrence and transience and collect
the known criteria which distinguish between these two cases.  In
Section \ref{slln} we establish strong laws of large numbers, i.e.\
the existence of a $P_0$-a.s.\ limit $v$ of $X_n/n$ as
$n\to\infty$. Section \ref{sbal} is devoted to the question whether
$v=0$ or $v\ne 0$. We also discuss in this section in more detail the
connection to branching processes with migration mentioned above,
which is a useful tool when $d=1$. In Section
\ref{flt} we consider one-dimensional and functional limit laws for
convergence in distribution. In the final section we quote some
results which did not fit into any of the previous sections.

Sections \ref{srt}-\ref{flt} are divided into two subsections
each. The first subsection deals with the one-dimensional situation, the
second one considers the multi-dimensional case.
\vspace*{5mm}

{\bf Notation.} For $k\in\Z$ we set
$T_k:=\inf\{n\ge 0\mid X_n=k\}$.
We need some notation for the environment which is
left over after the ERW has eaten some of its cookies. For any
$\om\in\Om$, $J\in\N_0\cup\{\infty\}$ and $(x_j)_{j<J}\in(\Z^d)^J$,
\[\psi(\om,(x_j)_{j< J})(z,e,i):=\om(z,e,i+\#\{0\le j<J \mid x_j=z\})\]
defines the environment $\psi(\om,(x_j)_{j< J})$ obtained from $\om$
by following the path $(x_j)_{j<J}$ and removing the currently first
cookie each time a site is visited. In this context we sometimes need  an independent copy of the ERW $X$, which we denote by $X'$.
 By $U\stackrel{\rm d}{=}V$ we
mean that $U$ and $V$ have the same distribution. Convergence in distribution is denoted by $\Rightarrow$. Saying that
$A\subseteq B$ $P$-a.s.\ means $P[A\backslash B]=0$. We write
$E[Z,A]:=E[Z\cdot \won_A]$ for random variables $Z$ and events $A$.
The integer part of $t\in\R$ is denoted by $[t]$.
\section{Finite or infinite range}\label{rara}
To the best of our knowledge the question whether the range $\{X_n\mid n\ge 0\}$ of the ERW is finite or infinite has not been considered in the literature yet.
In this section we give necessary and sufficient criteria for the range to be finite. In particular, under natural conditions the range is either $P_0$-a.s.\ finite or  $P_0$-a.s.\ infinite. This will be useful in Section \ref{srt}. First, we introduce some notation.
\begin{defi}\label{too}{\rm
For $x\in\Z^d$ and $e\in\mathcal E$ 
write
$x\stackrel{\om}{\to} x+e$ if and only if $\sum_{i\ge 1}\om(x,e,i)=\infty.$
Define $b_F:=
\PP[\forall e\in F: 0\not \stackrel{\om}{\to}e]$ for $F\subseteq\mathcal E$.
For $e\in\mathcal E$ write $b_e$ instead of $b_{\{e\}}$.
The transitive closure  in $\Z^d$ of the relation $\stackrel{\om}{\to}$ 
is denoted by $\stackrel{\om}{\to}$ as well. 
Moreover, $\mathcal C_x:=\{y\in\Z^d\mid x\stackrel{\om}{\to}y\}$. 
}
\end{defi}
The meaning of $x\stackrel{\om}{\to}y$ is illustrated by the following lemma, which follows from the Borel-Cantelli lemma.
\begin{lemma} \label{io}{\bf ($d\ge 1$)}
Let $\om\in\Om$ and  $x,y\in\Z^d$ with  $x\stackrel{\om}{\to}y$.
Then on the event that the ERW visits $x$ infinitely often, $y$ is $P_{0,\om}$-a.s.\ visited infinitely often as well.
\end{lemma}
\begin{theorem}\label{supa}{\bf ($d=1$, range)}
Assume {\rm (SE)} and {\rm (EL)}.
\begin{enumerate}
\item[(a)] If $b_1>0$ and $b_{-1}>0$ then 
the range is $P_0$-a.s.\ 
finite.
\item[(b)] If $b_1=0$ and $b_{-1}>0$ then $P_0$-a.s.\ $X_n\to+\infty$ as $n\to\infty$.
\item[(c)] If $b_1>0$ and $b_{-1}=0$ then $P_0$-a.s.\ $X_n\to-\infty$ as $n\to\infty$.
\item[(d)] If $b_1=0$ and $b_{-1}=0$ then the range is $P_0$-a.s.\ 
infinite.
\end{enumerate}
\end{theorem}
As we shall see in the proof in case (a) 
the walker eventually gets stuck between two essentially reflecting barriers of cookies. A barrier to his left is reflecting to the right and a barrier to his right is reflecting to the left. In case (b) there are only barriers reflecting to the right. They act like valves.  Once the walker has passed any of them from left to right he has a positive probability never to penetrate it in the opposite direction.
Case (d) is the richest case. 
\begin{lemma}\label{4} {\bf ($d=1$)} Assume {\rm (SE)}.
\begin{eqnarray}
&&\label{pos} \mbox{Let $b_{-1}>0$  and assume {\rm (EL)}. Then $P_0[\forall n\ge 0: X_n\ge 0]>0$}\\
&&\mbox{and  $P_0[\inf_{n\ge 0}X_n>-\infty]=1$.}\nonumber \\
&&\label{sub} \mbox{If $P_0[\forall n\ge 0: X_n\ge 0]>0$ then $P_0$-a.s.\ $\{\sup_{n\ge 0}X_n=\infty\}\subseteq \{X_n\to\infty\}$.}
\end{eqnarray}
\end{lemma}
\begin{proof}[Proof of Lemma \ref{4}]
Assume $b_{-1}>0$. If all the jumps from 0 go to 1 then $X_n\ge 0$ for all $n\ge 0$. Therefore, 
\[P_{0}[\forall n\ge 0:\ X_n\ge 0]\ \ge \ \EE[\pi_0],\quad\mbox{where}\quad
\pi_z:=\prod_{i\ge 1}\om(z,1,i)\ =\ \prod_{i\ge 1}(1-\om(z,-1,i)).
\]
By $b_{-1}>0$ and {\rm (EL)} we have $\EE[\pi_0]>0$, which yields the first statement.
Similarly, if $\inf_n X_n=-\infty$ then the  walk jumps for each $z\le 0$ at least once from $z$ to $z-1$. Therefore, for all $\om\in\Om$,
\begin{equation}\label{1st}
P_{0,\om}\left[\inf_{n\ge 0}X_n=-\infty\right]\ \le\ \prod_{z\le 0}(1-\pi_z)\ \le\ \exp\bigg(-\sum_{z\le 0}\pi_z\bigg).
\end{equation}
By assumption (SE) the sequence $(\pi_z)_{z\le 0}$ is stationary and ergodic.
Since $\EE[\pi_0]>0$ the right-hand side of (\ref{1st}) vanishes $\PP$-a.s.\ due to (SE). Hence $P_0$-a.s.\ $\inf_nX_n>-\infty$.

For the proof of (\ref{sub}) let $k,m\in\N$, $k\le m$ and let $(\F_n)_{n\ge 0}$ be the filtration generated by $X$. 
Then for all $\om\in\Om$, $P_{0,\om}$-a.s.\ on the event $\{\sup_nX_n=\infty\}$,
\begin{eqnarray}\nonumber
P_{0,\om}\left[\liminf_{n\to\infty} X_n\ge k\ \big|\  \F_{T_m}\right]&\ge&
 P_{0,\om}\left[\forall n\ge T_m:  X_n\ge m\ \big|\  \F_{T_m}\right]\\
&=&P_{m,\psi(\om,(X_n)_{n< T_m})}[\forall n\ge 0: X'_n\ge m]\label{vid} 
\end{eqnarray}
by the strong Markov property for the Markov chain $((X_i)_{0\le i\le
  n})_{n\ge 0}$.  However, since $\psi(\om,(X_n)_{n< T_m})(x)$ and
$\om(x)$ differ only at sites $x<m$ the expression in (\ref{vid}) is
equal to $P_{m,\om}[\forall n\ge 0: X_n\ge m]$ which is, due to (SE),
$\PP$-a.s.\ for infinitely many $m\in \N$ larger than
$\eps:=P_0[\forall n\ge 0: X_n\ge 0]/2$.  Hence, on the event
$\{\sup_nX_n=\infty\}$, $P_0$-a.s.,
\[\eps\le \liminf_{m\to\infty}P_{0,\om}\left[\liminf_{n\to\infty} X_n\ge k\ \big|\  \F_{T_m}\right]\ =\ \won_{\{\liminf_{n\to\infty} X_n\ge k\}}
\]
by Levy's 0-1 law.  Consequently, $P_0$-a.s.\
$\{\sup_nX_n=\infty\}\subseteq \{\liminf_{n\to\infty} X_n\ge
k\}$. Letting $k\to\infty$ then yields the claim.
\end{proof}
\begin{proof}[Proof of Theorem \ref{supa}]
  Statement (a) follows from (\ref{pos}) and its corresponding
  counterpart for the case $b_1>0$.  For the proof of (b) observe that
  by (\ref{pos}) $P_0$-a.s.\ $\inf_n X_n>-\infty$. Therefore,
  $P_0$-a.s.:
  \[\{X_n\to\infty\}^c\subseteq\bigcup_{z\in\Z}\{X_n=z\
  \mbox{i.o.}\}\stackrel{\rm L.\ \ref{io}}{\subseteq}\left\{\sup_{n\ge
      0}X_n=\infty\right\}
  \stackrel{(\ref{pos}),(\ref{sub})}{\subseteq}\{X_n\to\infty\},\]
  which implies claim (b). Claim (c) follows from (b) by symmetry.  In
  case (d), $P_0$-a.s.,
\[\left\{\sup_{n\ge 0}|X_n|<\infty\right\}\subseteq\bigcup_{z\in\Z}\{X_n=z\ \mbox{i.o.}\}\stackrel{\rm L.\ \ref{io}}{\subseteq}
\bigg\{\sup_{n_\ge 0}X_n=\infty\bigg\}\subseteq\bigg\{\sup_{n_\ge 0}|X_n|=\infty\bigg\},
\]
which yields the claim.
\end{proof}
If one strengthens the assumption (SE) to (IID) then a statement similar to Theorem \ref{supa} can be made also in higher dimensions. The following result  builds upon  \cite[Lem.\  2.2, 2.3]{HS11}, which deals with RWRE with possibly forbidden directions. It puts the example given  in \cite[Rem.\  1]{Zer06} in a general framework.
\begin{theorem}\label{R2}{\bf ($d\ge 1$, range)} Assume {\rm (IID)}  and {\rm (EL)}. If there is an orthogonal set $F\subset \mathcal E$ such that $b_F=0$ then the range is $P_0$-a.s.\ infinite. If there is no such set then the range is $P_0$-a.s.\ finite. 
\end{theorem}
\begin{lemma}\label{hs}{\bf ($d\ge 1$, \cite[Lem.\  2.3]{HS11})}
$\PP[\#\mathcal C_0=\infty]=1$ if and only if there is an orthogonal set $F\subset \mathcal E$ such that $b_F=0$.
\end{lemma}
For completeness we include a proof of this lemma.
\begin{proof}[Proof of Lemma \ref{hs}]
If there is an orthogonal set $F\subset \mathcal E$ with $b_F=0$ then there is $\PP$-a.s.\ a nearest-neighbor path $(y_n)_{n\ge 0}$ starting at $y_0=0$ with $y_n\stackrel{\om}{\to}y_{n+1}$ and $y_{n+1}-y_n\in F$ for all $n\ge 0$. Since $F$ is orthogonal this path is self-avoiding and hence $\{y_n\mid n\ge 0\}\subseteq \mathcal C_0$ is infinite.
Conversely, assume that there is no such set $F$. Then 
\[\left\{\#\mathcal C_0<\infty\right\}\supseteq\left\{\mathcal C_0\subseteq\{0,1\}^d\right\}\supseteq\left\{\forall x\in\{0,1\}^d\ \forall e\in F_x: x\not\stackrel{\om}{\to}x+e\right\},\]
where $F_x:=\{(-1)^{x_i+1}e_i\mid i=1,\ldots,d\}$ is the (orthogonal) set of directions pointing from $x$ towards the complement of $\{0,1\}^d$. By assumption $b_{F_x}>0$ for all $x\in\{0,1\}^d$ and therefore, by independence, $P[\#\mathcal C_0=\infty]<1$.
\end{proof}
\begin{proof}[Proof of Theorem \ref{R2}]
Following the idea behind \cite[Lem.\  2.2]{HS11} we shall show that the range is $P_0$-a.s.\ infinite if $\PP$-a.s.\ $\#\mathcal C_0=\infty$ and $P_0$-a.s.\ finite otherwise. 
The claim of Theorem \ref{R2} then follows from Lemma \ref{hs}. 

To prove the first implication we assume that $\mathcal C_0$ and hence all $\mathcal C_x$, $x\in\Z^d$, are $\PP$-a.s.\ infinite. Then $P_0$-a.s.,
\begin{eqnarray*}
\left\{\#\{X_n\mid n\ge 0\}<\infty\right\}&\subseteq&\bigcup_{x\in\Z^d}\left\{X_n=x\ \mbox{i.o.}\right\}
\stackrel{\rm L.\ \ref{io}}{\subseteq}\bigcup_{x\in\Z^d}\left\{C_x\subseteq\{X_n\mid n\ge 0\}\right\}\\
&\subseteq&\left\{\#\{X_n\mid n\ge 0\}=\infty\right\}.
\end{eqnarray*}
This proves the first implication.
For the opposite implication assume $\PP[\#\mathcal C_0<\infty]>0$. Choose
 $S\subset \Z^d$ finite such that $\PP[\mathcal C_0=S]>0$. By (EL) and Definition \ref{too} there is $\gamma\in\N$, strictly larger than the $\|\cdot\|_\infty$-diameter of $S$, such that 
$\PP[B_0]>0$, where for $x\in\Z^d$,
\begin{eqnarray*}
B_x&:=&\bigg\{\mathcal C_x=S+x,\ \sum_{y\in\mathcal C_x, i\in\N, e\in\mathcal E: y+e\notin \mathcal C_x}
\om(y,y+e,i)\le \ga,\\
&&\quad \forall y\in\mathcal C_x, i\in\N: \sum_{e\in\mathcal E: y+e\notin \mathcal C_x}
\om(y,y+e,i)\le 1-\ga^{-1}\bigg\}.
\end{eqnarray*}
Choose for all $e\in\mathcal E$ some $s_e\in S$ which minimizes $S\ni x\mapsto x\cdot e$. Furthermore, define the increasing sequence of stopping times $\si_{k}:=\inf\{n\in\N\mid \|X_n\|_{\infty}=k\ga\}\le \infty$, $k\in\N$. Then for all $k\ge 0$, by partitioning,
\begin{eqnarray}\label{11}
P_0[\si_{k+1}<\infty]&=&\sum_{e\in\mathcal E, x:\|x\|_\infty=x\cdot e=k\ga}P_0\left[\si_{k+1}<\infty, X_{\si_k}=x, B_{x-s_{e}}\right]\\
&&+ \sum_{e\in\mathcal E, x:\|x\|_\infty=x\cdot e=k\ga} P_0\left[\si_{k+1}<\infty, X_{\si_k}=x, B^c_{x-s_{e}}\right]\label{22}
\end{eqnarray}
First observe that for all  $e\in\mathcal E$ and  $x\in\Z^d$ with $\|x\|_\infty=x\cdot e=k\ga$,
\begin{equation}\label{cho}
k\ga\le \|z\|_\infty< (k+1)\ga\quad \mbox{for all $z\in S+x-s_{e}$.}
\end{equation}
The summands in (\ref{22}) are easy to handle:
\[P_0\left[\si_{k+1}<\infty, X_{\si_k}=x, B^c_{x-s_{e}}\right]\le 
\EE\left[P_{0,\om}[\si_{k}<\infty, X_{\si_k}=x], B_{x-s_{e}}^c\right].
\]
Observe that $P_{0,\om}[\si_{k}<\infty, X_{\si_k}=x]$ is $\si(\om(z,\cdot,\cdot); \|z\|_\infty<k\ga)$-measurable, whereas $B_{x-s_{e}}^c\in\si(\om(z,\cdot,\cdot); \|z\|_\infty\ge k\ga)$ by the first inequality in (\ref{cho}). Therefore, by (IID), the expression in (\ref{22}) can be estimated from above by
\begin{equation}\label{piano}
 \sum_{e\in\mathcal E, x:\|x\|_\infty=x\cdot e=k\ga}\EE\left[P_{0,\om}[\si_{k}<\infty, X_{\si_k}=x]\right]\ \PP\left[B_{x-s_{e}}^c\right]=P_0[\si_k<\infty]\ \PP[B^c_0].
\end{equation}
We now turn to the right-hand side of (\ref{11}). Due to the second inequality in (\ref{cho}) its summands can be estimated from above by
\begin{eqnarray}\label{bell}
&&P_0\left[\si_k<\infty, X_{\si_k}=x,  B_{x-s_{e}}, \{X_n: n\ge \si_k\}\not\subseteq S+x-s_{e}\right]\\
&=&E_0\left[ P_{x,\psi(\om,(X_n)_{n<\si_k})}\left[\{X'_n: n\ge 0\}\not\subseteq S+x-s_{e}\right],\si_k<\infty, X_{\si_k}=x,  B_{x-s_{e}}\right],\nonumber
\end{eqnarray}
where we used  the strong Markov property. By the first inequality in (\ref{cho}) we may replace in the last expression $\psi(\om,(X_n)_{n<\si_k})$ by $\om$ and get that the quantity on the right-hand side of (\ref{bell}) is equal to
\begin{eqnarray}\nonumber
&&\EE\left[P_{0,\om}\left[\si_k<\infty, X_{\si_k}=x\right]  P_{x,\om}\left[\{X_n: n\ge 0\}\not\subseteq S+x-s_{e}\right],  B_{x-s_{e}}\right]\\
&\stackrel{\rm (IID)}{=}&
P_0\left[\si_k<\infty, X_{\si_k}=x\right]\  \EE\left[P_{0,\om}\left[\{X_n: n\ge 0\}\not\subseteq S-s_{e}\right],  B_{-s_{e}}\right].\label{ping}
\end{eqnarray}
On the event $B_{-s_{e}}$ any walker who starts inside the cluster  $S-s_e=\mathcal C_{-s_e}$ at 0 and eventually leaves this cluster has to cross a bond $(y,y+f)$ with $y\in \mathcal C_{-s_e}$ and $y+f\notin \mathcal C_{-s_e}$. Therefore, on $B_{-s_{e}}$,
\begin{eqnarray*}
\lefteqn{P_{0,\om}\left[\{X_n\mid n\ge 0\}\not\subseteq S-s_{e}\right]\le
1-\prod_{y\in \mathcal C_{-s_e}, i\in\N}\Bigg(1-\sum_{f\in\mathcal E: y+f\notin C_{-s_e}}\om(y,y+f,i)
\Bigg)}\\
&\le& 1-\exp\Bigg(-c\sum_{y\in \mathcal C_{-s_e}, i\in\N, f\in\mathcal E: y+f\notin C_{-s_e}}\om(y,y+f,i)\Bigg)\ \le\ 1-e^{-c\ga},
\end{eqnarray*}
where  $c$ is a finite constant such that $\ln(1-x)\ge -cx$ for all $0\le x\le 1-\ga^{-1}$.
Summarizing, we get from (\ref{bell}) and (\ref{ping}) that the sum in (\ref{11}) is at most
$P_0\left[\si_k<\infty\right]\left(
1-e^{-c\ga}\right)\PP[B_0]$.
Together with (\ref{piano}) this implies by induction
\[P_0[\si_{k+1}<\infty]\le (1-e^{-c\ga}\PP[B_0])P_0[\si_k<\infty]\le (1-e^{-c\ga}\PP[B_0])^k.\]
Hence there is $P_0$-a.s.\ some $k$ with $\si_k=\infty$, which implies the finiteness of the range.
\end{proof}

\section{Recurrence and  transience}\label{srt}
Since ERW is not a Markov chain the standard definitions of recurrence and transience do not apply right away  in this setting.
\begin{defi}\label{d1}
{\rm
We call an ERW $X$ satisfying (SE) \textit{recurrent} if 
it visits $P_0$-a.s.\ every site $z\in\Z^d$ infinitely often. 
It is called \textit{transient} if $P_0$-a.s.\ no site is visited infinitely often, i.e.\ if
$P_0$-a.s.\ $|X_n|\to\infty$ as $n\to\infty$. For any direction $\ell\in\R^d\backslash\{0\}$ we say that the ERW is \textit{transient in direction $\ell$} if $P_0$-a.s.\ $X_n\cdot\ell\to\infty$ as $n\to\infty$. 
We set 
\[A_\ell:=\left\{\lim_{n\to\infty} X_n\cdot\ell=\infty\right\}.\]
 In $d=1$ transience in direction 1 (resp.\ -1) is also called \textit{transience to the right (resp.\ left)}.
}
\end{defi}
\subsection{Results for $d=1$}
\begin{theorem}\label{01a} {\bf ($d=1$)}
Assume {\rm (SE)} and {\rm (EL)}. \\
Then 
$ P_0[|X_n|\to\infty]=P_0[A_1\cup A_{-1}]\in\{0,1\}.$
\end{theorem}
The proof uses the following lemma. This type of lemma is standard, see e.g.\ \cite[Lem.\  9]{Zer06} for references.
\begin{lemma}\label{D}{\rm ($d=1$)}
Assume {\rm (SE), (WEL),} and $P_0[A_1]>0$. Then also 
\begin{equation}\label{lamp}
P_0[\{\forall n\ge 0: X_n\ge 0\}\cap A_1]>0.
\end{equation}
\end{lemma}
\begin{proof}[Sketch of the proof of Lemma \ref{D}] 
\begin{figure}[t]   
    \psfrag{n}{$n$}
    \psfrag{x}{$X_n$}
    \psfrag{z}{$z$}
    \epsfig{figure=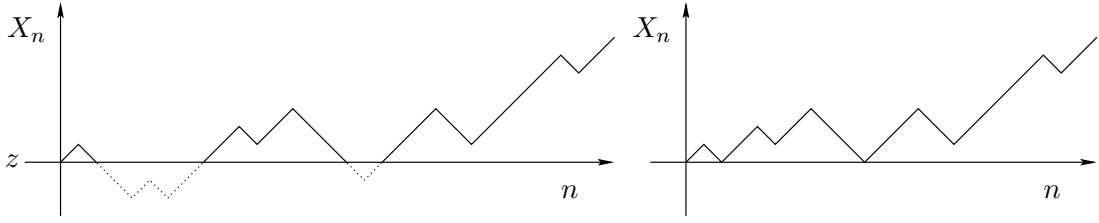,height=29mm}
    \caption{For the proof of (\ref{hm})}
    \label{D0}
  \end{figure}
  First, one shows that for all $z\in\Z$ and $\om\in\Om$ such that
  $\om(z,1,i)>0$ for all $i\ge 1$ and $P_{z,\om}[A_1]>0$ one also has
\begin{equation}\label{hm}
P_{z,\om}[\{\forall n\ge 0: X_n\ge z\}\cap A_1]>0.
\end{equation}
The proof of this can be done along the lines of the proof of
\cite[Lem.\ 8]{Zer05}.  One simply erases the (finitely many)
excursions from $z$ to the left of $z$ and visits the sites to the
right of $z$ in the same order as before, see Figure \ref{D0}.
For the proof of (\ref{lamp}) note that by assumptions (SE) and (WEL)
there is $\PP$-a.s.\ some $z\ge 0$ such that
\begin{eqnarray*}
  0&<&P_0[A_1,\ \forall i\in\N: \om(z,1,i)>0]\\
  &=&E_0\left[P_{z,\psi(\om,(X_n)_{n<T_z})}[A_1],\ T_z<\infty, \ \forall i\in\N: \om(z,1,i)>0\right].
\end{eqnarray*}
 Therefore, by (\ref{hm}),
\begin{eqnarray}\nonumber
0&<& E_0\left[P_{z,\psi(\om,(X_n)_{n<T_z})}[\{\forall n\ge 0: X'_n\ge z\}\cap A_1],\ T_z<\infty\right]
\\
&\le& P_z[\{\forall n\ge 0: X_n\ge z\}\cap A_1]\label{xx}
\end{eqnarray}
since $\psi(\om,(X_n)_{n<T_z})(x,\cdot,\cdot)=\om(x,\cdot,\cdot)$ for $x\ge z$. Due to (SE) the right-hand side of (\ref{xx}) is equal to $P_0[\{\forall n\ge 0: X_n\ge 0\}\cap A_1]$.
\end{proof}
\begin{proof}[Proof of Theorem \ref{01a}] Since the ERW has bounded jumps we have $P_0$-a.s.\ $\{|X_n|\to\infty\}=A_1\cup A_{-1}$. Now  we consider the four cases of Theorem \ref{supa}. In cases (a), (b) and (c) the 0-1 statement  is obvious due to Theorem \ref{supa}.  For case (d) we assume $b_1=b_{-1}=0$ and $P_0[|X_n|\to\infty]>0$. Without loss of generality we assume $P_0[A_1]>0$.
 By Lemma \ref{D} and (EL), $P_0[\forall n\ge 0: X_n\ge 0]>0$.
 Therefore, $P_0$-a.s.,
\[
\{|X_n|\to\infty\}^c\subseteq \bigcup_{z\in\Z}\{X_n=z\ \mbox{i.o.}\}\stackrel{L.\ \ref{io}}{\subseteq}
\left\{\sup_{n\ge 0}X_n=\infty\right\}\stackrel{(\ref{sub})}{\subseteq}
A_1\subseteq\{|X_n|\to\infty\}.
\]
This implies $P_0[|X_n|\to\infty]=1$.
\end{proof}
\begin{cor} {\rm ($d=1$)}
Assume {\rm (SE)} and {\rm (EL)}. Then the ERW is either recurrent or transient or has $P_0$-a.s.\ finite range.
\end{cor}
\begin{proof}
We consider again the four cases of Theorem \ref{supa}. In  cases (a)-(c) the claim of the corollary is obviously true. For case (d) assume $b_1=0=b_{-1}$.  According to Theorem \ref{01a}, $P_0[|X_n|\to\infty]$ is either 0 or 1. In the latter case the walk is transient. In the former case there is $P_0$-a.s.\ some $z\in\Z$ which is visited infinitely often. By Lemma \ref{io} all other sites are $P_0$-a.s.\ visited infinitely often as well, i.e.\ the walk is recurrent. 
\end{proof}
\begin{problem}\label{op1}{\bf ($d=1$)} {\rm 
The event in Theorem \ref{supa}  that the range is infinite can be rephrased as $\{\sup_n|X_n|=\infty\}$. Dropping there and in Theorem \ref{01a} the absolute values raises the following problem: find conditions which imply
\begin{eqnarray}\label{si}
 P_0\left[\sup_{n\ge 0}X_n=\infty\right],\ P_0\left[\inf_{n\ge 0}X_n=-\infty\right]&\in&\{0,1\}\qquad\mbox{and/or}\\
\label{01e}
P_0[X_n\to\infty],\ P_0[X_n\to-\infty]&\in&\{0,1\}.
\end{eqnarray}
Two such sets of conditions are given in Theorem \ref{rt} below.  It
is not difficult to see, using Theorems \ref{01a} and \ref{supa} and
Lemma \ref{io}, that under the assumptions {\rm (SE)} and {\rm (EL)}
the two 0-1 statements (\ref{si}) and (\ref{01e}) are equivalent. See
also Problem \ref{arx}.  }
\end{problem}
Next we present some criteria for recurrence or transience.  
First we consider the case $M(z)\le 1$, $z\in\Z,$ where the walk is not getting excited about the sites which it has visited before.
\begin{prop}\label{m1} {\bf $(d=1)$}
Assume {\rm (SE)}, {\rm (WEL)} and $\PP$-a.s.\  $M(0)\le 1$. Then
the ERW is recurrent.
\end{prop}
One could prove this statement by combining a monotonicity argument as in Proposition \ref{mono} below (see e.g.\ \cite[Th.\  16]{Zer05})  with Theorem \ref{rt} (\ref{r1}) below.
However, we present here a more direct proof in the spirit of  \cite[Section 2]{BW03}. 
\begin{lemma}\label{cos}
Let $(Z_n)_{n\in\N}$ be a stationary and ergodic sequence of non-negative random variables with $E[Z_1]>0$. Then $\sum_nZ_n/n=\infty$ almost surely.
\end{lemma}
\begin{proof}[Proof of Lemma \ref{cos}]
By the ergodic theorem one can recursively choose positive integers $(n_k)_{k\in\N}$ such that for all $k\ge 1$, $n_{k}\ge s_{k-1}:=n_1+\ldots+n_{k-1}$ and 
\[P[B_k^c]\le k^{-2},\quad\mbox{where}\quad 
B_k:=\bigg\{\frac{1}{n_k}\sum_{n=s_{k-1}+1}^{s_k} Z_n\ge E[Z_1]/2\bigg\}.\]
By the Borel-Cantelli lemma there is a random $K\in\N$ such that almost surely all events $B_k,\ k\ge K,$ 
occur. Hence, since $Z_n\ge 0$ almost surely,
\[
\sum_{n\ge 1}\frac{Z_n}{n}\ \ge\ \sum_{k\ge K}\sum_{n=s_{k-1}+1}^{s_{k}}\frac{Z_n}{n}\ \ge\ 
\sum_{k\ge K}\frac{n_{k}}{s_{k}}\ \frac{1}{n_{k}}\sum_{n=s_{k-1}+1}^{s_{k}}Z_n
\ \ge\ \sum_{k\ge K}\frac{E[Z_1]}{4}\ =\ \infty.
\]
\end{proof}
\begin{proof}[Proof of Proposition \ref{m1}]
By Theorem \ref{supa} (d), $P_0$-a.s.\ $\sup_{n\ge 0}|X_n|=\infty$. Therefore,
\begin{eqnarray}\label{ben}
P_0[\{X_n=0\ \mbox{i.o.}\}^c]&\le& P_0\left[\left\{X_n=0\ \mbox{i.o.}\right\}^c\cap\left\{\sup_{n\ge 0}X_n=\infty\right\}\right]\\
&&+\ P_0\left[\left\{X_n=0\ \mbox{i.o.}\right\}^c\cap\left\{\inf_{n\ge 0}X_n=-\infty\right\}\right].\label{ben2}
\end{eqnarray}
The term on the right-hand side of (\ref{ben})  is equal to  
\begin{eqnarray}\nonumber
&&P_0\left[\bigcup_{L\in\N}\{\forall k\ge L: T_k<\infty\}\cap\{\forall n\ge T_L: X_n>0\}\right]\\
&=&\lim_{L\to\infty}\EE\Bigg[P_{0,\om}[T_L<\infty]\prod_{k\ge L}P_{k,\om'_k}[T_{k+1}<T_0]\Bigg],
\label{we}
\end{eqnarray}
where $\om'_k:=\psi(\om,(0,1,\ldots,k-1))$.
Here we used that $P_{0,\om}$-a.s.\ $P_{k,\psi(\om,(X_n)_{n<T_k})}[T_{k+1}<T_0]=P_{k,\om'_k}[T_{k+1}<T_0]$. 
Observe that 
\[P_{k,\om'_k}[T_{k+1}<T_0]
=1-\om(k,-1,1)P_{k-1,\om'_{k+1}}[T_{k+1}>T_0]=1-\frac{2\om(k,-1,1)}{k+1}\]
by the gambler's ruin problem for the simple symmetric RW on $\Z$. Therefore, Lemma \ref{cos} applied to $Z_{n+1}:=2\om(n,-1,1)$ yields that the infinite product in (\ref{we}) $\PP$-a.s.\ vanishes. Therefore, the term on the right-hand side of (\ref{ben}) is zero. By symmetry the same holds for the expression in (\ref{ben2}). Consequently, by Lemma \ref{io}, any $z\in\Z$ is $P_0$-a.s.\ visited infinitely often. 
\end{proof}
\begin{remark}{\bf ($d=1$, first return time to the origin)} {\rm
\cite[Section 3.3]{AR05} deals with model (BW) for $d=1$ and any $p\in (0,1)$ and employs a physical approach to show  that $P_1[T_0>n]\sim n^{p-1}$ as $n\to\infty$. }
\end{remark}
If one allows more than just one cookie per site then the walk may become transient. Whether this  happens or not depends under certain conditions on the parameter 
\[\delta:=\EE\Bigg[\sum_{e\in\mathcal E, i\ge 1}\omega(0,e,i)e\Bigg]\]
if it exists as an element of $(\R\cup\{\pm\infty\})^d$. (We shall see in Sections \ref{sbal} and \ref{flt} that for $d=1$ the
parameter $\delta$ characterizes other phase transitions as well.) To
interpret $\delta$, observe that after consuming a cookie
$\om(z,\cdot,i)$ the walk is displaced on average by
$\sum_e\om(z,e,i)e$, which we call the \textit{drift} stored in that
cookie. Thus the parameter $\delta$ is, when it exists, the expected
total drift stored in a cookie stack.
\begin{remark}\label{d0}{\bf ($d=1$, w.l.o.g.\ $\delta\ge 0$)} {\rm 
Observe that for $d=1$ replacing $\omega$ by $1-\omega$ switches the sign of $\delta$  and that for each $\om$  the distribution of $X$ under $P_{0,1-\om}$ is the same as of $(-X_n)_{n\ge 0}$ under $P_{0,\om}$. Hence when investigating the behavior of $X$ one may restrict oneself to the case $\delta\ge 0$. 
}
\end{remark}
\begin{theorem}\label{rt}{\bf ($d=1$, recurrence, transience \cite[Th.\ 12]{Zer05} \cite[Th.\ 1]{KZ08})}\\
  Assume either
\begin{equation}\label{r1}
\mbox{{\rm (SE), (POS$_1$)} and
  $\PP[\om(0,1,\cdot)=(1,1/2,1/2,\ldots)]<1$} \qquad \text{or}
\end{equation}
\begin{equation}\label{r2}
\mbox{{\rm (IID), (BD),} and {\rm (WEL).}} 
\end{equation}
Then the ERW is recurrent if  
$\delta\in[-1,1]$, transient to the right if $\delta>1$
and transient to the left if $\delta<-1$.
\end{theorem}
Note that in the case $\PP[\om(0,1)=(1,1/2,1/2,\ldots)]=1$ one has trivially $P_0$-a.s.\ $X_n=n$, which makes the walk transient to the right although $\delta=1$.
Observe also that the above criterion for recurrence/transience of ERW is quite different from the one for one-dimensional RWRE due to Solomon, see e.g.\ \cite[Th.\ 2.1.2]{Zei04}. 
\begin{proof}[Idea of the proof of Theorem \ref{rt} in case (\ref{r1})] 
Let $\delta\ge 0$ and assume that the walk is either recurrent or transient to the right.
Consider
\[M_n:=X_n-D_n,\quad\mbox{where}\quad D_n:=\sum_{m=0}^{n-1}\left(2\om\left(X_m,1,\#\{k\le m\mid X_k=X_m\}\right)-1\right)
\]
is the total drift stored in the cookies which have been consumed by the walker before time $n$. 
By the Doob-Meyer decomposition of the submartingale $(X_n)_{n \ge 0}$ the process $(M_n)_{n\ge 0}$ is a $P_{0,\om}$-martingale w.r.t.\ its canonical filtration.
A naive application of the optional stopping theorem  yields that the expected drift $E_{0,\om}[D_{T_k}]$ stored in the cookies which have been consumed by the walk by time $T_k$, $k\ge 0$, is equal to $E_{0,\om}[X_{T_k}]=k$. We now compare this quantity to the expected total drift $k\delta$ stored in all the cookies $\om(z,\cdot,i), i\ge 1, 0\le z<k$.
 
If the walk is transient to the right then it consumes only  finitely many cookies to the left of 0. Hence, if $\delta<1$ then even consuming all the cookies between 0 and $k$ before time $T_k$ would not be enough, for $k$ large, to satisfy the required total demand $E_{0,\om}[D_{T_k}]=k$. Hence the walk cannot be transient to the right if $\delta<1$. 

Conversely, if $\delta>1$ then the walk cannot afford to return too often to 0 before time $T_k$ because otherwise it would eventually eat most of the cookies between 0 and $k-1$ and thus exceed its dietary restriction $E_{0,\om}[D_{T_k}]=k$. This indicates that the walk cannot be recurrent in this case. 

This argument can be made precise in case (\ref{r1}). An even less formal argument is given in \cite[p.\ 2569]{AR05}.
\end{proof}
The proof of Theorem \ref{rt} in case (\ref{r2}) relies on branching processes similar to the ones used in the proof of Theorem \ref{bal}, which is sketched below.
\begin{problem}{\rm
Can one replace in (\ref{r2}) the assumption (IID) by (SE)?
}
\end{problem}
\begin{problem}{\rm Compute in the transient case the probability $P_0[\forall n\ge 1 : X_n\ne 0]$ never to return to the starting point. See \cite[Th.\ 18]{Zer05} for an example.
}
\end{problem}

\begin{remark}{\bf (Strips)} {\rm
A result similar to Theorem \ref{rt} has been shown in \cite[Th.\ 2]{Zer06} for ERWs on strips $\Z\times\{0,1,\ldots,L-1\}, L\in\N,$ under assumptions similar to (\ref{r1}).}
\end{remark}

\begin{remark}\label{haveit} {\bf ($d=1$, RW in ``have your cookie and eat it" environments)} {\rm    
A self-interacting RW $Y=(Y_n)_{n\ge 0}$ on $\Z$ called RW in a ``have your cookie and eat it" environment has been introduced in \cite{P10}. There, ``at each site $x$, the probability of jumping to the right is $\om(x)\in[1/2,1)$, until the first time the process jumps to the left from site $x$, from which time onward the probability of jumping to the right" from that site is 1/2. Here the sequence  $(\om(x))_{x\in\Z}$ (with values in $[1/2,1)^\Z$) is assumed to be stationary and ergodic. 

Note that  if $\om(x)=q$ for $x<0$ the RW $Y$ behaves on the negative integers in the same way as the RW perturbed at extrema, which we  described in Section \ref{ext}. There we also showed how a RW perturbed at extrema can be viewed as an ERW. The same applies to $Y$ with the difference that the measure $\PP$ which provides the environment for $Y$  does not lack spatial homogeneity but satisfies (SE). 
Note however that  neither (POS$_1$) nor  (BD) nor (WEL) are fulfilled by this $\PP$. Nevertheless, although $\PP$ does not meet the requirements of Theorem \ref{rt} the conclusion of Theorem \ref{rt} is still true for this $\PP$ as stated in \cite[Th.\ 2]{P10}.
 See also Remark \ref{p2}.
}
\end{remark}

\begin{remark}{\bf ($d=1$, RWRE as underlying process)}
{\rm In \cite{B} the simple symmetric RW as underlying process is replaced by a RWRE which is transient, say, to the left. On the first $M_z\ge 0$ visits to $z\in\Z$ the ERW is deterministically pushed to $z+1$, only on later visits to that site the RWRE environment takes effect. Here the random environment and $(M_z)_{z\in\Z}$ are i.i.d.\ with $\PP[M_0=0]>0$. Sufficient criteria for transience to the left, recurrence or transience to the right of the resulting ERW are given in terms of the tail of the distribution of $M_0$. }
\end{remark}
\subsection{Results for $d\ge 2$}
The following result is the multidimensional analogue to Theorem \ref{01a}. It  generalizes Kalikow's zero-one law  for directional transience of  multidimensional RWRE as stated in \cite[Prop.\  3]{ZM01}, see also \cite[Th.\ 3.1.2]{Zei04} and \cite[Th.\ 1.3]{HS11} for versions with different hypotheses.
\begin{theorem}\label{kal2} {\bf ($d\ge 1$, Kalikow-type zero-one law)}\\
Assume {\rm (IID)}, {\rm (EL)} and let $\ell\in\R^d\backslash\{0\}$. Then 
$ P_0[|X_n\cdot\ell|\to\infty]=P_0[A_\ell\cup A_{-\ell}]\in\{0,1\}.$
\end{theorem}
The proof follows the one given in  \cite[Prop.\  3]{ZM01}. It uses the following two lemmas. The first one is the multidimensional analogue of Lemma \ref{D}. 
\begin{lemma}\label{D2}{\rm ($d\ge 1$)}
Assume {\rm (SE), (EL),} $\ell\in\R^d\backslash\{0\}$, and $P_0[A_\ell]>0$. Then
$P_0[\{\forall n\ge 0: X_n\cdot \ell\ge 0\}\cap A_\ell]>0.$
\end{lemma}
The proof of this lemma  is similar to the proof of Lemma \ref{D}
outlined above and is identical to that of \cite[Lem.\  9]{Zer06}. (The general assumption (UEL) in \cite{Zer06} is not needed for the proof of \cite[Lem.\  9]{Zer06} and can be replaced by (EL).)

Let us denote for any interval $I\subseteq \R$ by $S_I$ the slab $\{x\in\Z^d: x\cdot \ell\in I\}$.
The second lemma gives conditions under which the ERW cannot visit any slab $S_{[u,w]}$, $u<w,$ infinitely often without ever visiting both neighboring half spaces. 
\begin{lemma}\label{slab}{\bf ($d\ge 1$)}
Assume {\rm (IID)}, $\EE[\om(0,e,1)]>0$ for all $e\in\mathcal E$ and let $\ell\in\R^d\backslash\{0\}$ and $u,w\in\R$ with $u<w$. Furthermore assume that the range of the ERW is $P_0$-a.s.\ infinite. Then
\begin{equation}\label{oma}
P_0[\{X_n\cdot\ell\ge u\ {\rm i.o.}\}\cap\{\forall n\ge 0: X_n\cdot\ell\le w\}]=0.
\end{equation} 
\end{lemma}
\begin{proof}[Proof of Lemma \ref{slab}] 
The proof is a bit more involved than the one of the corresponding statement \cite[Lem.\  4]{ZM01} for RWRE.
Denote by $A$ the event considered in (\ref{oma}) and let $F$ be the event that the walker visits only finitely many distinct elements of $S_{[u,w]}$. It is enough to show that $P_0[A\cap F]=0=P_0[A\cap F^c]$.

First, we consider $P_0[A\cap F^c]$.
Without loss of generality assume $\ell\cdot e_1>0$.  Then there is $N\in\N$ such that for all $x\in S_{[u,w]}$ we have $(x+Ne_1)\cdot\ell> w$.
On the event $A\cap F^c$ the walker visits infinitely many sets $S_{[u,w]}\cap (\Z\times \{y\})$, $y\in\Z^{d-1},$ since each such set is finite. Each time the walker visits such a set for the first time he has, due to (IID), independently of his past the $P_0$-probability $\EE[\om(0,e_1,1)]^N$ to walk in the next $N$ steps in direction $e_1$ thus reaching the half space $S_{(w,\infty)}$. Having infinitely many independent such chances the walker will $P_0$-a.s.\ not miss all of them. Therefore, $P_0[A\cap F^c]=0$.

Our treatment of $P_0[A\cap F]$ deviates from the corresponding step in the proof of \cite[Lem.\  4]{ZM01}. The proof is by contradiction. Assume that $P_0[A\cap F]>0$  and recall Definition \ref{too}. Then there is some $k\in\N$ such that $P_0[A\cap F_k]>0$, where $F_k$ is the event that the walker (a) visits at most $k$ elements of $S_{[u,w]}$, (b) has infinite range, (c) does not cross any directed edge $(y,z)$ infinitely often unless $y\stackrel{\om}{\to}z$, and (d) visits every element of $\mathcal C_y$ infinitely often whenever $y\in\Z^d$ is visited infinitely  often. (The events in (b)-(d) have full $P_0$-measure.)

Denote for $x\in\Z^d$  by $B_x\subseteq \Om$ the event that there is a self-avoiding nearest-neighbor path $(y_n)_{n\ge 0}$ starting at $y_0=x$ such that for all $n\in\N$,
\begin{eqnarray}
 x&\in&\mathcal C_{y_n},\label{to}\\
 \mathcal C_{y_n}&\subseteq& S_{(-\infty,x\cdot\ell]}\quad\mbox{ and}\label{to2}\\
\#\left(\mathcal C_{y_n}\cap S_{\{x\cdot\ell\}}\right)&\le& k.\label{to3}
\end{eqnarray}
If $B_x$ occurs denote by $(y_n(x))_{n\ge 0}$ a path with these properties. Choose it according to some deterministic rule if there are several such paths.

On the event $A\cap F_k$ there is at least one random vertex $x$ which is visited infinitely often and maximizes $x\cdot\ell$. The maximal number of such vertices is at most $k$ due to (a).
For each such $x$ the event $B_x$ occurs. Indeed, by K\"onig's lemma \cite[Lemma A]{Koe27} for directed graphs and (b) there is a random self-avoiding nearest-neighbor path $(y_n)_{n\ge 0}$ starting at $y_0=x$ such that all its directed edges $(y_n,y_{n-1}),\ n\in\N,$ are crossed infinitely many times. 
Any such path satisfies   (\ref{to})-(\ref{to3}) for all $n\in\N$. (\ref{to}) holds since 
$y_m\stackrel{\om}{\to}y_{m-1}$ for all $m\in\N$ due to (c).
Moreover,
since $x$ maximizes $y\mapsto y\cdot \ell$ among all vertices which are visited infinitely often (d)  implies (\ref{to2}). And since the walker visits at most $k$ vertices which maximize $y\mapsto y\cdot \ell$ we also have (\ref{to3}).

Therefore, $0<P_0[A\cap F_k]\le \PP[\bigcup_x B_x]$. Hence $\PP[B_x]>0$ for some $x\in\Z^d$. Due to (IID),  $(\won_{B_x})_{x\in\Z^d}$ is stationary  under $\PP$ w.r.t.\ the shifts on $\Z^d$. Hence $\PP[B_0]=\PP[B_x]>0$. Choose $m\in\N$ with $m+1\ge 3k/\PP[B_0]$. Since  $(\won_{B_x})_{x\in\Z^d}$ is also ergodic there is by the ergodic theorem, see e.g.\ \cite[Th.\ A.11.5]{El}, $\PP$-a.s.\ some random $L\in\N$ such that 
\begin{equation}\label{as}
(m+1)\sum_{x\in[-L,L]^d}\won_{B_x}\ >\ \frac{m+1}{2}\, \PP[B_0]\ \#[-L,L]^d\ \ge \ k\ \#[-L-m,L+m]^d.
\end{equation}
For all $n=0,\ldots,m$ and all $x\in[-L,L]^d$ for which $B_x$ occurs
we have $y_n(x)\in[-L-m,L+m]^d$ since $(y_n(x))_{n=0,\ldots,m}$ is a
nearest-neighbor path starting at $x$. By (\ref{as}) there are
strictly more than $k\ \#[-L-m,L+m]^d$ such pairs $(n,x)$. It follows from the 
pigeonhole principle that  there are $y\in[-L-m,L+m]^d$ and pairwise
distinct pairs
$(n_0,x_0),\ldots,(n_k,x_k)\in\{0,\ldots,m\}\times[-L,L]^d$ such that
$y=y_{n_i}(x_i)$ for all $i=0,\ldots,k$.  By (\ref{to}) and (\ref{to2}),
$x_0,\ldots,x_k\in\mathcal C_y\subseteq S_{(-\infty,\min_i\,
  x_i\cdot\ell]}.$ Therefore,
$x_0\cdot\ell=\ldots=x_k\cdot\ell$ and hence $x_0,\ldots,x_k\in\mathcal C_y\cap
S_{\{x_0\cdot\ell\}}$. However, since every path
$(y_{n}(x_i))_{n\ge 0},$ $i=0,\ldots,k,$ is self-avoiding the
sites $x_0,\ldots,x_k$ are pairwise distinct as well. But this
contradicts (\ref{to3}). 
\end{proof}

\begin{proof}[Sketch of the proof of Theorem \ref{kal2}] By Theorem
  \ref{R2} the range of the ERW is either finite or infinite. In the
  first case the statement is obvious. In the second case the proof
  goes along the same lines as the one of \cite[Prop.\ 3]{ZM01}. Since
  the increments of the ERW are uniformly bounded $P_0$-a.s.\ exactly one
  of the three events $A_\ell$ or $A_{-\ell}$ or
  $\bigcup_{u<w}A_{u,w}$ occurs, where $A_{u,w}:=\{X_n\cdot\ell\in [u,w]\ {\rm
    i.o.}\}$. Suppose that $P_0[A_\ell]>0$ and let $u<w$. By Lemma \ref{slab}
  we have $P_0$-a.s.\ $\sup_nX_n\cdot \ell=\infty$ on 
  $A_{u,w}$.  However, each time the process $(X_n\cdot\ell)_n$
  reaches a new maximum $x>w$ it has by Lemma \ref{D2} and (IID) the
  same positive chance never to fall back again below $x$. Consequently, after
  a geometric number of trials $(X_n\cdot\ell)_n$ has reached a level
  larger than $w$ below which it will never fall again. Hence,
  $P_0$-a.s.\ $A_{u,w}\subseteq A^c_{u,w}$, i.e.\ $P_0[A_{u,w}]=0$.
\end{proof}

The main result of \cite{BW03} is that under assumption (BW) the ERW
is transient in direction $e_1$ whenever $d\ge 2$. For the proof, one
first couples the ERW $X$ in the way described in Section \ref{ss11}
to a simple symmetric RW $Y=(Y_n)_{n\ge 0}$. Then one considers
so-called \textit{tan points} of $Y$. These are sites $x\in\Z^d$
which are visited by $Y$ prior to any other point on the ``sun
ray" $\{x+ke_1\mid k\in\N\}$.
Every time when $Y$ reaches a tan point, $X$ reaches a fresh site and
eats a non-placebo cookie which pushes it in direction $e_1$. Showing
that $Y$ has enough tan points then implies transience of $X$ in
direction $e_1$.

However, it seems difficult to adapt the method of tan points to other
settings in which the drift is not along a coordinate direction or the
excitement occurs at later visits. However, by combining the
martingale approach explained after Theorem \ref{rt} and the method of
the environment viewed from the particle this result was extended in
\cite[Th.\ 1]{Zer06} to the following more general result.
\begin{theorem}\label{rt2}{\bf ($d\ge 2$, directional transience)}\\
Assume {\rm (IID)} and {\rm (UEL)} and let $\ell\in\R^d\backslash\{0\}$ such that 
{\rm (POS$_\ell$)} holds with
$\delta\cdot \ell>0$.
Then the ERW is transient in direction $\ell$, i.e.\ $P_0[A_\ell]=1$.
\end{theorem}

\begin{problem}\label{arx}{\rm {\bf ($d\ge 2$)} Find conditions which
    imply the zero-one law $P_0[A_\ell]\in\{0,1\}$ for all
    $\ell\in\R^d\backslash\{0\}$ (cf.\ Problem \ref{op1}).  Are e.g.\
    (IID), (BD) and (UEL) sufficient (just as for $d=1$, see
    (\ref{r2}) of Theorem \ref{rt})?  

For RWRE this zero-one law holds
    under (IID) and (EL) if $d=2$, see \cite[Th.\ 1]{ZM01}. According
    to \cite[Th.\ 1.5]{HS11} the assumption (EL) can be dropped. For
    $d\ge 3$ it is an important open problem, see e.g.\ \cite[Open
    Problem 1.4]{DR10}.
}
\end{problem}

Not much is known in general 
about recurrence and transience in the sense of Definition \ref{d1} for $d\ge 2$.
\begin{problem} {\rm ($d\ge 2$)
Assume (IID) and (UEL). Is the ERW either recurrent or transient?}
\end{problem}
\begin{rp}\label{berw}{\bf ($d\ge 2$, balanced ERW)} {\rm The work
    \cite{BKS11} introduces so-called $M(d_1,d_2)$-RW. This is an ERW on $\Z^{d_1+d_2}$  which upon the first visit to a vertex performs a $d_1$-dimensional simple symmetric RW step within the first $d_1$-coordinates and upon later visits to that same vertex performs a simple symmetric RW step within the last $d_2$-coordinates. It is  proved that the $M(2,2)$-RW is transient and conjectured that the  $M(1,2)$- and the $M(2,1)$-walk are transient as well while the $M(1,1)$-walk is believed to be recurrent.

More generally, consider measures $\PP$ which are balanced in the sense that $\PP$-a.s.\ $\om(z,e,i)=\om(z,-e,i)$ for all $z\in\Z^d, e\in\mathcal E, i\in\N$. It follows from recent work \cite[Th.\ 1.2, Prop.\  1.4]{PPS} that any balanced ERW satisfying (UEL) in $d\ge 3$ returns $P_0$-a.s.\ only finitely often to 0 if there are at most $\max\{2,(d-1)/2\}$ cookies $\om_1,\ldots,\om_k\in\M_{\mathcal E}$
such that for all $z\in\Z^d$ and $i\in\N$ there is $\PP$-a.s.\ some $j\in\{1,\ldots,k\}$ such that $\om(z,\cdot,i)=\om_j$.

One may wonder, whether any balanced ERW satisfying  (IID) (or (SE)) and (UEL) 
is recurrent if $d=2$ and transient if $d\ge 3$.  For RWRE this is
true, see \cite[Th.\ 3.3.22]{Zei04}.  }
\end{rp}
\begin{remark}{\bf ($d=3$, ERW against a wall)}
{\rm \cite{ABK08} deals with an ERW in a special spatially non-homogeneous environment $\om$ on $\Z^3$. The walk starts at 0, is pushed down whenever it reaches a new site, but is prohibited to enter the lower half space (i.e.\ $\om((x,y,z),-e_3,1)=1$ and $\om((x,y,0),-e_3,i)=0$ for all $x,y,z\in\Z, z\ne 0, i\in\N$) and otherwise behaves like a simple symmetric RW. It is  shown, in particular,  that this walk returns a.s.\ infinitely often to its starting point. 

  A similar statement regarding recurrence of an ERW which is
  ``excited to the origin" is made in \cite[(1)]{K07}. In \cite{K07}
  one can find related open problems.  }
\end{remark}

\section{Strong law of large numbers}\label{slln}
We say that the ERW $X$ satisfies a strong law of large numbers if there is
a non-random $v\in\R^d$, called the {\em velocity} or {\em speed} of the walk, such that
\[\lim_{n\to\infty}\frac{X_n}{n}=v\quad\mbox{$P_0$-a.s..}
\]
\subsection{Results for $d=1$}
\begin{theorem}\label{lln}{\bf ($d=1$, law of large numbers)}\\
If {\rm (SE)} and {\rm (POS$_1$)} hold then $X$ satisfies a strong law of large numbers with speed $v\in[0,1]$.
 If {\rm (SE)} and the 0-1 law {\rm (\ref{si})} hold then 
$X$ satisfies a strong law of large numbers with speed $v\in[-1,1]$.
\end{theorem}
\begin{proof} For the first statement see \cite[Th.\ 13]{Zer05}.
For the second statement observe that although the proof of \cite[Prop.\ 13]{KZ08} has been stated under stronger conditions it provides a proof of the second
claim if $P_0\left[\sup_nX_n=\infty\right]=1$ or
$P_0\left[\inf_nX_n=-\infty\right]=1$. In the remaining case, when
both probabilities are 0, one trivially has $P_0$-a.s.\
$X_n/n\to 0$.
\end{proof}
Next we discuss some properties of the speed as a function of $\PP$.
 We
    call a cookie $\om_2(z,\cdot,i)$ \textit{stronger} than another
    cookie $\om_1(z,\cdot,i)$ if it pushes its consumer more to the
    right, i.e.\ if $\om_2(z,1,i)\ge \om_1(z,1,i)$. One should expect
    that making cookies stronger does not decrease the speed of the walk. More precisely, an environment $\om_2$ is
    called stronger than another environment $\om_1$ if each cookie
    $\om_2(z,\cdot,i)$ in $\om_2$ is stronger than the corresponding
    cookie $\om_1(z,\cdot,i)$ in $\om_1$. A probability measure
    $\PP_2$ on $\Om$ is called stronger than another probability
    measure $\PP_1$ if there is a probability measure on
    $\{(\om_1,\om_2)\in\Om^2\mid\mbox{$\om_2$ is stronger than
      $\om_1$}\}$ with first marginal $\PP_1$ and second marginal
    $\PP_2$. The following statement is contained in \cite[Th.\ 17]{Zer05} under the additional assumption (POS$_1$).
\begin{prop}\label{mono} {\bf ($d=1$, monotonicity of $v$)} If $\PP_2$ is stronger than $\PP_1$ and if the ERW satisfies under both measures $\EE_i[P_{0,\om}[\cdot]]$, $i=1,2$, a law of large numbers with speed $v_i$  then $v_2\ge v_1$.
\end{prop}
\begin{proof} First, we show that for $i=1,2$,
\begin{equation}\label{frac}
\lim_{k\to\infty}\frac{k}{T_k}=\max\{v_i,0\}\qquad\mbox{$\EE_i[P_{0,\om}[\cdot]]$-a.s..}
\end{equation}
If $\sup_n X_n=\infty$ then the sequence $(k/T_k)_{k\ge 0}$ is a subsequence of $(X_n/n)_{n\ge 0}$.
Therefore, $\EE_i[P_{0,\om}[\cdot]]$-a.s.\
$\{\sup_{n} X_n=\infty\}\subseteq\{\lim_{k}k/T_k=v_i\ge 0\}.$
On the other hand, if $\sup_n X_n<\infty$ then $T_k=\infty$ for some $k$ and thus
$\EE_i[P_{0,\om}[\cdot]]$-a.s.\
$\{\sup_{n} X_n<\infty\}\subseteq\{\lim_{k}k/T_k=0\ge v_i\}.$
In any case (\ref{frac}) holds. From this one can conclude like in the proof of \cite[Th.\ 17]{Zer05} that
$\max\{v_2,0\}\ge \max\{v_1,0\}$. (The general assumption (POS$_1$) present in \cite{Zer05} is not needed for this conclusion.) Analogously, by considering $T_{-k}$ instead of $T_k$ we obtain $\min\{v_2,0\}\ge \min\{v_1,0\}$. Hence $v_2\ge v_1$.
\end{proof}
\begin{remark}{\rm {\bf
($d=1$,  monotonicity of $v$ w.r.t.\ swapping cookies)} In \cite[Section 5]{HS12} yet another monotonicity property for $v$ is discussed. It is shown that ``one
cannot decrease the (lim sup)-speed of a cookie RW by swapping stronger cookies in a
pile with weaker cookies that appear earlier in the same pile (and doing this at each site)".
For the precise statement we refer to 
\cite[Th.\ 5.1]{HS12}.
}\end{remark}
\begin{remark}{\rm 
{\bf
 ($d=1$,  continuity of $v$)} In \cite{BS08a} deterministic environments $\om$ of the form $\om(z,1,\cdot)=(p_1,\ldots,p_M,1/2,1/2,\ldots)$ are considered, where $M\in\N$ and $\bar p=(p_1,\ldots,p_M)\in[1/2,1)^M$ are fixed. It is shown in \cite[Th.\ 1.2]{BS08a} that the corresponding speed $v$ depends continuously on $\bar p$. It is plausible that the method of proof and result may work under the assumptions (IID), (BD) and (WEL) as well.}
\end{remark}
\subsection{Results for $d\ge 1$}
In the (IID) setting a regeneration structure, which goes back to \cite{KKS75} and has been heavily used to study RWRE, see e.g.\ \cite[Section 3.2]{Zei04}, can be straightforwardly extended to include ERW, see e.g.\ \cite[p.\ 114, Rem.\  3]{Zer05} and \cite[Section 3]{BR07}. See also Figure \ref{fig:1} for $d=1$.
\begin{figure}[h]
    \centering
    \includegraphics[height=6cm]{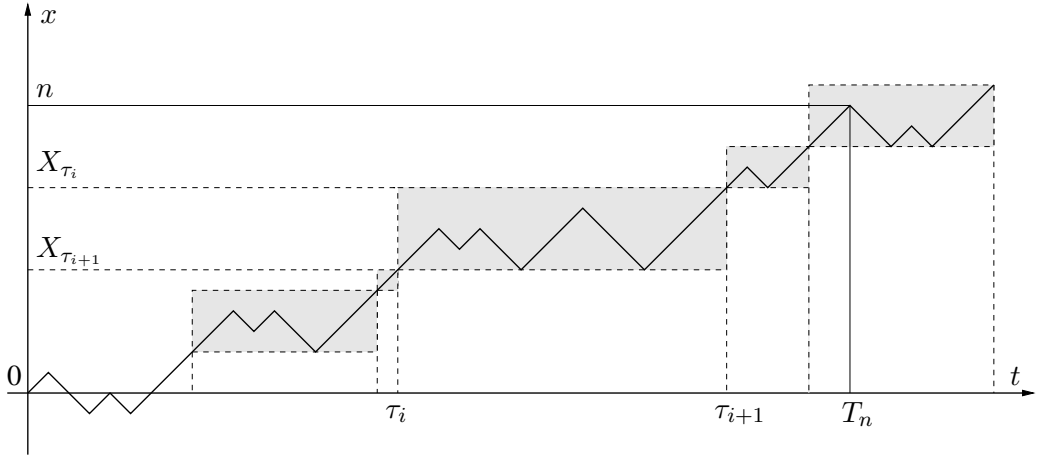}
    \caption{Regeneration structure: sizes and contents of the shaded boxes are i.i.d..}
    \label{fig:1}
  \end{figure}
\begin{lemma}\label{rs}{\bf ($d\ge 1$, regeneration structure)} Assume {\rm (WEL)} if $d=1$ and {\rm (EL)} if $d\ge 2$. Furthermore
assume {\rm (IID)}, let $\ell\in\R^d\backslash\{0\}$, and assume $P_0[A_\ell]>0$. Then there are $P_0[\ \cdot\mid A_\ell]$-a.s.\ infinitely many random times $n\ge 0$, so-called {\em regeneration times}, such that $X_m\cdot \ell<X_n\cdot \ell$ for all $m<n$ and $X_m\cdot \ell\ge X_n\cdot \ell$ for all $m\ge n$. Call the increasing enumeration of these times $(\tau_k)_{k\in\N}$. Then the random $\bigcup_{n\in\N}(\Z^d)^n$-valued vectors $(X_n)_{0\le n\le \tau_1}$, $(X_n-X_{\tau_i})_{\tau_i\le n\le \tau_{i+1}}\ (i\ge 1)$ are independent w.r.t.\ $P_0[\ \cdot\mid A_\ell]$. Moreover, the vectors $(X_n-X_{\tau_i})_{\tau_i\le n\le \tau_{i+1}}\ (i\ge 1)$ have the same distribution under $P_0[\ \cdot\mid A_\ell]$ as $(X_n)_{0\le n\le \tau_1}$ under $P_0[\ \cdot\mid \forall n\ X_n\cdot \ell\ge 0]$. Also  $E_0[(X_{\tau_2}-X_{\tau_1})\cdot\ell\mid A_\ell]<\infty$.
\end{lemma}
Note that the assumption (UEL) present in \cite[Section 3.2]{Zei04} is not needed in Lemma \ref{rs}. Instead the ellipticity conditions needed are inherited from Lemma \ref{D} for $d=1$ and Lemma \ref{D2} for $d\ge 2$.
This regeneration structure can be used in the same way as for RWRE, see e.g.\ \cite[(3.2.8)]{Zei04}, to prove the following directional law of large numbers. 
\begin{theorem}\label{lln2}{\bf ($d\ge 1$, directional law of large numbers)}\\
Under the assumptions of Lemma~\ref{rs} the following holds: $P_0$-a.s.\ on the event $A_\ell$,
\begin{equation}\label{rudi}
\lim_{n\to\infty}\frac{X_n\cdot \ell}{n}=v_\ell:=\frac{E_0[(X_{\tau_2}-X_{\tau_1})\cdot\ell\mid A_\ell]}{E_0[\tau_2-\tau_1\mid A_\ell]}\in[0,1].
\end{equation}
\end{theorem}
The challenge is to determine whether $v_\ell$ is strictly positive or not, i.e.\ whether $E_0[\tau_2-\tau_1\mid A_\ell]$ is finite or not, see Section \ref{sbal}.
\begin{problem}\label{op}{\bf ($d\ge 1$)} {\rm 
Assume (IID), (EL), $\ell\in\R^d\backslash\{0\}$.
Does $X_n\cdot \ell/n$ converge $P_0$-a.s.\ on $\{|X_n|\to\infty\}^c$ to 0 as $n\to\infty$? This would generalize \cite[Th.\ 1]{Zer02}.}
\end{problem}
\begin{theorem}\label{LLN}{\bf ($d\ge 1$, law of large numbers)}\\
Assume {\rm (WEL)} if $d=1$ and {\rm (EL)} if $d\ge 2$.
Furthermore
assume {\rm (IID)}
and let $\ell_1,\ldots,\ell_d$ be a basis of $\R^d$ such that 
$P_0[A_{\ell_i}\cup A_{-\ell_i}]=1$  
for all $i=1,\ldots,d$.
Then there are  $v\in\R^d$ and $c\ge 0$ such that $P_0$-a.s.
\[\lim_{n\to\infty}\frac{X_n}{n}\in\{v,-cv\}.\]
If, in addition,
$P_0[A_{\ell_i}]=1$  for all $i=1,\ldots,d$ then $X$ satisfies a strong law of large numbers with velocity $v\in\R^d$ such that $v\cdot \ell_i\ge 0$ for all $i=1,\ldots,d$.
\end{theorem}
\begin{proof}
The proof of the first statement is analogous to the proof of \cite[Th.\ 1.5]{DR10} and uses Theorem \ref{lln2}. The additional assumption 
$P_0[A_{\ell_i}\cup A_{-\ell_i}]=1$, which is not needed in \cite[Th.\ 1.5]{DR10}, is due to Problem \ref{op}.
The second statement follows directly from Theorem \ref{lln2}.
\end{proof}
\begin{remark}\label{MP}{\rm
Note that by Theorem \ref{rt2} the assumption $P_0[A_{\ell_i}]=1$ in Theorem \ref{LLN} is satisfied if (MPRV$_\ell$) holds for some $\ell\in\R^d\backslash\{0\}$ and $\ell_i$ is chosen from a sufficiently small neighborhood of $\ell$.
}\end{remark}
\begin{remark}\label{alp}{\bf ($d\ge 6$, LLN via cut points)} {\rm
    Using cut points,  a law of large numbers for RWRE has been  proved in 
\cite[Th.\ 1.4]{BSZ03} in the case where at least 5 coordinates of the walk jointly form a standard RW. This method was generalized in \cite[Th.\ 1.1]{HS}, see also \cite[Th.\ 2.1]{H}, to prove a law of large numbers for similar high-dimensional ERWs in the (IID) setting.}
\end{remark} 
\begin{remark}\label{exp}{\rm 
($d$ large, {\bf monotonicity, continuity and differentiability of $v\cdot e_1$}) For the original model (BW), it has been shown in \cite[Th.\ 1.1]{VdHH10} by lace expansion techniques,  that the first coordinate $v\cdot e_1$ is monotonically increasing in the drift parameter $p$ if $d\ge 9$.
The same approach also allows to prove a weak law of large numbers for $d\ge 6$, see \cite[Th.\ 2.3]{vdHH12}. 
\cite{H} assumes  (IID) and $\PP$-a.s.\ $\om(0,e_j,i)=1/(2d)$ for all $j=2,\ldots,d$. Using the same expansion technique, it is  shown in \cite[Th.\ 2.3]{H} that the velocity is in an appropriate sense continuous in the drift parameters $\EE[\om(0,e_1,i)]$ if $d\ge 6$ and even differentiable if $d\ge 8$. Strict monotonicity for $d\ge 12$ is also considered. For similar statements regarding monotonicity and continuity in the context discussed in Remark \ref{alp} see  
\cite[Th.\ 1.2]{HS}.
}
\end{remark}
\section{Ballisticity}\label{sbal}
An ERW is called \textit{ballistic} if it satisfies a strong law of large numbers with non-zero velocity $v$.  
\subsection{Results for $d=1$}
For recurrence and transience, the number 1, as a value for $|\delta|$
or $M(0)$, played a crucial role, see Theorem
\ref{rt} and Proposition \ref{m1}. Similarly, as we are about to see, the number 2 plays a special
role for ballisticity.
\begin{prop}{\bf ($d=1$, no speed with two cookies)}\label{M2}
Assume {\rm (SE)} and  $M(0)\le 2$ $\PP$-a.s.\ and  the ellipticity condition $\PP[\om(0,1,1)<1, \om(1,1,1)<1], \PP[\om(0,-1,1)<1, \om(1,-1,1)<1]>0$. If the ERW satisfies a strong law of large numbers with speed $v$ then $v=0$.
\end{prop}
\begin{proof} 
Under the additional assumption (POS$_1$) the strong law of large numbers with speed $0$ has been proven in \cite[Th.\ 19]{Zer05}. Replacing cookies with negative drift by placebo cookies and using Proposition \ref{mono} therefore yields $v\le 0$. By symmetry also $v\ge 0$.
\end{proof}
\begin{theorem}\label{bal}{\bf ($d=1$, (non-)ballisticity \cite[Th.\ 1.1, 1.3]{MPV06}, \cite[Th.\ 1.1]{BS08a}, \cite[Th.\ 2]{KZ08})}
Assume {\rm (IID), (BD)}, and {\rm (WEL)}.  Then $v<0$ if $\delta<-2$, $v=0$ if $\delta\in[-2,2]$ and $v>0$ if $\delta>2$.
\end{theorem}
\begin{problem}{\rm Can one replace in Theorem \ref{bal} the assumption (BD) by (POS$_1$)?}
\end{problem}
\begin{proof}[Idea of the proof  of Theorem \ref{bal}] The proof  relies on a connection with branching
processes, which we now describe. The main ideas go back to
\cite[Sec.\ 6]{Ha52},  \cite{Kn63}, 
\cite{KKS75}, and \cite{T96}. They were introduced into the study of ERW in \cite{BS08a}. We follow the exposition given in \cite{KM11}.

By Remark \ref{d0} we may assume without loss of generality that
$\delta\ge 0$. It will be convenient to use the following coin-tossing
construction of $X$ as characterized in (\ref{top}).  Let
$(\Om',\mathcal{F'})$ be some measurable space equipped with a family
of probability measures $P_{0,\omega},\ \omega\in\Omega$, and $\pm
1$-valued random variables $B_{i}^{(k)},\ k\in\Z,\ i\in\N,$ which are
for all $\omega\in\Omega$ independent under $P_{0,\omega}$ with
distribution given by
$P_{0,\omega}[B_{i}^{(k)}=\pm 1]=\omega(k,\pm 1,i).$
The events $\{B_{i}^{(k)}=1\}$ (resp.\ $\{B_{i}^{(k)}=-1\}$), $i\in\N$,
$k\in\Z$, will be called ``successes'' (resp.\ ``failures''). 
Then an ERW $X$ starting at 0 in
the environment $\omega$ can be defined on the
probability space $(\Om',\mathcal{F'},P_{0,\omega})$ by 
\[
X_0:=0\quad\mbox{and}\quad X_{n+1}:=X_n+B_{\#\{0\le r\le n\mid X_r=X_n\}}^{(X_n)}\quad \mbox{for}\quad n\ge 0.
\]
By Theorem~\ref{rt}, $P_0[T_n<\infty]=1$ for
all $n\in\N$. 
For $n\in\N$ and $k\in\Z$ let 
\[J_{n,k}^\downarrow:=\sum_{j=0}^{T_n-1}\won_{\{X_j=k,\ X_{j+1}=k-1\}}\quad\mbox{and}\quad
J_{n,k}^\uparrow:=\sum_{j=0}^{T_n-1}\won_{\{X_j=k,\ X_{j+1}=k+1\}}
\]
 be the number of jumps from
$k$ to $k-1$ resp.\ from $k$ to $k+1$ before time $T_n$.
Observe that for all $n\in\N$ and $k\in\Z$,
\begin{equation}\label{st}
J_{n,k}^\uparrow=J_{n,k+1}^\downarrow+\won_{0\le k<n}.
\end{equation}
Therefore,
\[T_n=\sum_{k\in\Z}J_{n,k}^\uparrow+J_{n,k}^\downarrow\stackrel{(\ref{st})}{=}n+2\sum_{k\le n}J_{n,k}^\downarrow=n+2\sum_{0\le k\le n}J_{n,k}^\downarrow+2\sum_{k<0}J_{n,k}^\downarrow.\]
The last sum is bounded above by the total time spent by $X$ below $0$.
When $\delta>1$, i.e.\ $X$ is transient to the right, the
    time spent below $0$ is $P_0$-a.s.\ finite, and the growth of
    $T_n$ is determined by $J^\downarrow_{n,k}$ for $0\le k\le n$.
Thus, it is enough to consider the ``backward'' process
$\left(J_{n,n}^\downarrow,J_{n,n-1}^\downarrow,\ldots,J_{n,0}^\downarrow\right)$. Obviously, $J_{n,n}^\downarrow=0$
for every $n\in\N$. 

Moreover,
denote by $F^{(k)}_m$ the number of
``failures'' in the sequence $B^{(k)}$ before the $m$-th
``success''.  Since the last departure from $k$ before time $T_n$ leads to $k+1$ we have
\[J^\downarrow_{n,k}=F^{(k)}_{J^\uparrow_{n,k}}\stackrel{(\ref{st})}{=}F^{(k)}_{J^\downarrow_{n,k+1}+1}\] for all $0\le k<n$. Since the sequence $((F_m^{(k)})_{m\in\N})_{k\in\Z}$ is i.i.d.\ under $P_0$ 
we conclude that the
distribution of $\left(J_{n,n}^\downarrow,J_{n,n-1}^\downarrow,\ldots,J_{n,0}^\downarrow\right)$ under $P_0$ 
coincides with that of $(V_0,V_1,\dots,V_n)$, where $V=(V_k)_{k\ge
  0}$ is the Markov chain defined by
\[
V_0:=0,\quad V_{k+1}:=F^{(k)}_{V_k+1}\quad \mbox{for $k\ge 0$}.
\]
Thus it suffices to study $V$. Observe that $V$ is a
branching process with the following properties:
(i) $V$ has exactly $1$ immigrant in each generation. The 
immigration occurs before the reproduction.
(ii) One can enumerate the individuals in the $k$-th generation in such a way that 
the number $\zeta_m^{(k)}$ of offspring of the $m$-th individual in
  generation $k$ is given by the number of failures
  between the $(m-1)$-th and $m$-th success in the sequence
  $B^{(k)}$. (Here the time of the 0-th success is 0.) In particular, if $V_{k}\ge M$ then the offspring
  distribution of each individual after the $M$-th one is
 $\mathrm{Geom(1/2)}$ (i.e.\ geometric on $\{0\}\cup\N$ with parameter
$1/2$). 

Therefore,
we can write for $k\ge 0,$
\begin{equation}
 \label{defV}
V_{k+1}=\sum_{m=1}^{M\wedge (V_k+1)}
\zeta^{(k)}_m+\sum_{m=1}^{V_k-M+1}\xi^{(k)}_m,
\end{equation} 
where $\xi^{(k)}_m,\,k\ge 0,\,m\ge 1,$ are i.i.d.\
$\mathrm{Geom(1/2)}$ random variables and the vectors $\left(\zeta^{(k)}_1,
  \zeta^{(k)}_2, \ldots, \zeta^{(k)}_M\right)$, $k\ge 0$, are i.i.d.\
under $P_0$ and independent of $(\xi^{(k)}_m)_{k\ge 0,m\ge 1}$.
For each $k\ge 0$ the random variables $\zeta^{(k)}_m, m=1,\ldots,M$, are in general
neither independent nor identically distributed.  Define the length $\sigma^V$ of
the life cycle and the total progeny $A^V$ over a single life
cycle by
\begin{equation*}
\sigma^V:=\inf\{j\ge 1\,|\, V_j=0\}
\quad\mbox{and}\quad A^V:=\sum_{j=0}^{\sigma^V-1}V_j.     
\end{equation*} 
These two quantities partially characterize the regeneration
structure of transient ERW (see Lemma~\ref{rs} with $d=1$ and
\cite[Lem.\  12, (30)]{KZ08}), namely,
\begin{equation}
  \label{corr}
  \sigma^V\overset{\mathrm{d}}{=}X_{\tau_2}-X_{\tau_1},\quad
  A^V\overset{\mathrm{d}}{=}\frac{\tau_2-\tau_1-(X_{\tau_2}-X_{\tau_1})}{2}.
\end{equation}
Detailed information about the tail behavior of $\sigma^V$ and
$A^V$ is the key to the proof of Theorem~\ref{bal} as well as to
the results about scaling limits of $T_n$ and $X_n$ (see
Section~\ref{flt}).
  There is a large number of papers which study the extinction
  probabilities of branching processes with migration. Unfortunately,
  they do not seem to include this particular setting. In \cite{KZ08}
   existing results were adopted by introducing
  appropriate modifications of the process $V$. In \cite{BS08a} most and in \cite{KM11} all of the
  required results about branching processes were obtained
  directly. 
\begin{theorem}{\bf (\cite[Th.\ 2.1, 2.2]{KM11})}\label{X}
 Assume {\rm (IID), (BD),} and {\rm (WEL)} and let $\delta>0$. Then 
there are constants $C_1,C_2\in(0,\infty)$ such that 
\[\lim _{n \rightarrow \infty} n^\delta P_0(\sigma^V>n) =
C_1\ \ \text{and}\ \ \lim_{n\to\infty}n^{\delta/2}P_0\left(A^V>n\right)=C_2.\]
\end{theorem}
For future reference in Section \ref{flt} we shall record the following corollary of Theorem \ref{X} and (\ref{corr}) for transient
ERW.
\begin{cor}\label{halifax}
  Assume {\rm (IID), (BD),} and {\rm (WEL)} and let $\delta>1$. Then
there are constants $C_1,C_3\in(0,\infty)$ such that 
\begin{align}
\label{tailx}
\lim _{n \rightarrow \infty} n^\delta P_0(X_{\tau_2}-X_{\tau_1}>n) =
C_1,\\
\label{tailt}
  \lim_{n\to\infty}n^{\delta/2}P_0\left(\tau_2-\tau_1>n\right)=C_3.
\end{align}
\end{cor}
Combining (\ref{tailt}) with Theorem~\ref{lln2} we
obtain Theorem~\ref{bal}.
\end{proof}
\begin{proof}[Idea of the proof of Theorem \ref{X}]
The main point
is that the process $V$ killed upon reaching $0$ is reasonably well
described by a simple diffusion. The parameters of such a diffusion
can be easily computed at the heuristic level. For $V_k\ge M$,
(\ref{defV}) implies that
\begin{equation*}
V_{k+1}-V_k=\sum_{m=1}^M \zeta^{(k)}_m-M+1+\sum_{m=1}^{V_k-M+1}(\xi_m^{(k)}-1).
\end{equation*}
By 
\cite[Lem.\  3.3]{BS08a} or \cite[Lem.\ 
17]{KZ08},
\begin{equation*}
E_0\left[\sum_{m=1}^M \zeta^{(k)}_m-M+1\right]=1-\delta.
\end{equation*}
The term $\sum_{m=1}^M \zeta^{(k)}_m-M+1$ is independent of 
$\sum_{m=1}^{V_k-M+1}(\xi_m^{(k)}-1)$. When $V_k$ is large, the latter is, conditioned on $V_k$,
approximately normal with mean $0$ and variance essentially equal to $2V_k$. 
Therefore, the relevant diffusion should be given by the following
stochastic differential equation:
\begin{equation*}
dY_t=(1-\delta)\,dt+\sqrt{2 Y_t}\,dB_t,\quad Y_0=y>0, \quad t\in[0,\tau^Y],
\end{equation*}
where $\tau^Y=\inf\{t\ge 0\,|\,Y_t=0\}$.  Observe that $2Y_t$ is a
squared Bessel process of dimension $2(1-\delta)$. Thus, if $\delta>0$
then $\tau^Y<\infty$ a.s..  Using the scaling properties of $Y_t$,
it is not hard to compute the tail behavior of $\tau^Y$ and
$\int_0^{\tau^Y}Y_s\,ds$. The bulk of the work is to transfer these
results to $\sigma^V$ and $A^V$.
\end{proof}
\begin{remark}\label{p2} {\bf ($d=1$, ``have your cookie and eat it" RW)} {\rm According to \cite[Th.\ 3]{P10} the statement of Theorem \ref{bal} (except for possibly the critical case $|\delta|=2$) also holds for the measure $\PP$ which reflects ``have your cookie and eat it" environments as described in Remark \ref{haveit}.
} 
\end{remark}

\begin{re}{\rm {\bf ($d=1$, $\delta$ does not characterize ballisticity under assumption (SE))}
The following example shows that one cannot replace in Theorem \ref{bal} assumption (IID) by (SE), not even under the additional assumption (POS$_1$). The first such example in this context was constructed in \cite[p.\ 290]{MPV06}. 
Thus, while $\delta$ characterizes recurrence and transience for $d=1$ also in the general setting of (SE) it is not sufficient to  characterize ballisticity under this assumption.  

Let $1/2<p\le 1$.
Then there is a RWRE with speed $v=0$ in a stationary and ergodic environment which uses exactly $p$ and $1/2$ as  transition probabilities to the right. More formally,  there is a measure $\PP$ on $\Om$ which satisfies (SE) and 
\[
\PP[\forall i\in\N: \om(0,1,i)=p]=1-\PP[\forall i\in\N: \om(0,1,i)=1/2]>0
\]
such that  $P_0$-a.s.\  $X_n/n\to 0$ as $n\to\infty$.
(Note that (POS$_1$) and $\delta=\infty$ hold.) 

To construct such $\PP$ let $(B_z)_{z\in\Z}$ be a stationary and ergodic $\{0,1\}$-valued process on some probability space with probability measure $P$ such that $P[B_0=1]>0$ and $E[\tilde Z]=\infty$, where
$\tilde Z:=\inf\{z\in \N_0\mid B_{-z}=1\}$. For example, $(B_z)_{z\in\Z}$ could be a two-sided stationary discrete renewal process with i.i.d.\ inter-point distances which have an infinite second moment.
Then set for all $i\in\N$, $\wom(z,1,i)=p$ if $B_z=1$ and $\wom(z,1,i)=1/2$
if $B_z=0$ and let $\PP$ be the distribution of $\wom$ on $\Om$. 
It then follows from \cite[Th.\ 2.1.9 (c)]{Zei04} that $P_{0}$-a.s.\ $X_n/n\to 0$ as $n\to\infty$.

To construct an ERW example which also satisfies (BD) one can replace
all cookies $\om(z,\cdot,i)$ with $i$ larger than some fixed $M$ by
placebo cookies. By Proposition \ref{mono} this
cannot increase the speed. However, since (POS$_1$) still holds the
speed cannot become negative by Theorem \ref{lln} either and is therefore
still 0.  }
\end{re}
\begin{remark} {\rm ($d=1$, {\bf $v$ is
      not a function of $\delta$.}) While under the conditions of
    Theorem \ref{bal} the value of $\delta$ determines whether $v$ is
    positive, zero, or negative, it does not determine the value of
    $v$, not even under strong assumptions.  
For example, let $p\in(1/2,1]$ and $M\in\N$ such that $(2p-1)M>2$ and define 
$\om_{p,M}\in\Om$ by setting for all $z\in\Z$, $\om_{p,M}(z,1,i):=p$ if $i\le M$ and $\om_{p,M}(z,1,i):=1/2$ if
    $i>M$. Then the parameter $\delta$ corresponding to $\PP:=\delta_{\om_{p,M}}$ is equal to $(2p-1)M$ and therefore the corresponding ERW has speed $v>0$ due to Theorem \ref{bal}. Now pick $M'\in\N$ such that $M'>\delta/v$ and set $p':=(1+\delta/M')/2$. Then $\PP':=\delta_{\om_{p',M'}}$ has the same parameter $\delta'=\delta$ as $\PP$, yet its induced speed $v'$ is, due to Proposition \ref{mono}, not larger than the speed $v''=2p'-1=\delta/M'<v$ of the asymmetric
    nearest neighbor RW on $\Z$ which corresponds to $\PP'':=\delta_{\om_{p',\infty}}$.
}
\end{remark}
\subsection{Results for $d\ge 2$}
It was observed in \cite[Th.\ 5]{BW03} that ERW under assumption (BW)
is ballistic if $d\ge 4$ and left as an open question whether the same
holds true in dimensions 3 and 2. In two unpublished preprints this
question was answered in the affirmative by Kozma first for $d=3$ and
then also for $d=2$.  The paper \cite{BR07} uses a different method to
prove this statement. Its authors obtain a lower bound on the number of
tan points to show that the expected time $E_0[\tau_2-\tau_1]$ between
two successive regeneration times is finite which, by (\ref{rudi}),
implies ballisticity. 

However, as we already mentioned before Theorem
\ref{rt2}, methods involving tan points might not be easily adaptable
to more general settings. Using a related but more general approach
the following result about ballisticity was obtained in \cite{MPRV}.  Here we only
mention that \cite{MPRV} uses martingale techniques and completely
avoids tan points.
A key ingredient of the proof is a lower bound on the
growth of the range of the ERW \cite[Prop.\ 4.1]{MPRV}.  The positivity
condition plays a crucial role in the proofs. (Recall from Remark \ref{MP} that the strong law of large numbers holds in this setting.)
\begin{theorem}\label{bal2}{\bf ($d\ge 2$, ballisticity, \cite[Th.\ 1 (i)]{BR07}, \cite[Th.\ 1.2 (i)]{MPRV})}\\
Let $\ell\in \R^d\backslash\{0\}$ and assume {\rm (MPRV$_\ell$), (IID)}, and {\rm (UEL)}.
Then the ERW is ballistic and its velocity $v$  satisfies $v\cdot \ell>0$. 
\end{theorem}
According to \cite[Th.\ 1.1]{MPRV} this result can also be extended to certain processes which do not fit into the framework considered here. For more recent developments in this direction see \cite{MP}.
\begin{remark}\label{pnc}{\bf ($d\ge 2$, positive and negative
    cookies)} {\rm Little is known about ballisticity outside of the
    setting of Theorem \ref{bal2}.  In \cite[Section 9]{KZ08} an
    example is given for an ERW in $d\ge 4$ which is ballistic in
    spite of $\delta=0$. For $d\ge 9$, \cite[Th.\ 2.2]{H} shows, using
    lace expansion, that if the drift induced by the first cookie is
    strong enough the first cookie can overrule the subsequent cookies
    and determine the direction of the velocity. On the other hand,
    for $d\ge 2$, if the first cookie is weak enough and all the
    subsequent cookies induce a drift in the opposite direction then
    their influence might win, see \cite[Lem.\ 2.4]{H}.}
\end{remark}
\section{Limit laws and functional limit theorems}\label{flt}
In this section we discuss limit laws and functional limit theorems
for ERWs under the averaged measure $P_0$. Quenched limit laws are mentioned in  Problem and Example \ref{pee} at the end of this section. 

\subsection{Results for $d=1$}\label{flt1} 
There is a rich variety of limit laws and limit processes for ERW in
$d=1$. The limit processes obtained so far in the ERW literature are
Brownian motion \cite{KZ08}, a certain class of Brownian
motions perturbed at extrema \cite{Do11, DK}, the running maximum of
Brownian motion \cite{DK}, as well as, under a different kind of
scaling, so-called excited Brownian motions \cite{RS}. As a new result
we prove here functional limit theorems for the model of Section
\ref{1.4} with $\delta\in (1,2)\cup(2,4]$ with limit processes which
are strictly stable with indices $\alpha\in [1/2,2]$ and skewness
parameter 1 (totally skewed to the right). The proof is based on
\cite{Li78}.

Let us outline the plan of this subsection. First, we state functional
limit theorems in the recurrent case. In the transient regime, we
start with  convergence of one-dimensional distributions.  Then
we recall and compare the notions of convergence in $J_1$ and $M_1$ on
the Skorokhod space.
Further, we state and prove, using \cite{Li78}, functional limit theorems
for transient ERWs. Finally, we consider convergence to excited
Brownian motion.

Let $D[0,\infty)$ be the Skorokhod space, i.e.\ the space of
right-continuous real-valued functions with left limits on
$[0,\infty)$ endowed with the standard Skorokhod topology (i.e.\
$J_1$, see definition (\ref{dj}) below).  Denote by
$\overset{J_1}{\Rightarrow}$ the weak convergence on $D[0,\infty)$ in
$J_1$. Unless stated otherwise, we shall assume that all processes
start at the origin at time $0$.

We begin with the recurrent case and define the candidate for the
limiting process, the so-called  {\em Brownian motion perturbed at extrema}
(\cite{Da96,PW97,CD99,D99}).  Let $\alpha,\beta\in(-\infty,1)$ and let $B=(B(t))_{t\ge 0}$ denote a standard
Brownian motion.  
Then 
\begin{equation}
  \label{abp}
  X_{\alpha,\beta}(t)=B(t)+\alpha\sup_{s\leq t}
X_{\alpha,\beta}(s)+\beta \inf_{s\leq t} X_{\alpha,\beta}(s)
\end{equation}
has a pathwise unique solution $X_{\alpha,\beta}=(X_{\alpha,\beta}(t))_{t\ge 0}$ that is adapted to the
filtration of $B$ (see e.g.\ \cite[Th.\ 2]{CD99}). It is called $(\alpha, \beta)$-perturbed Brownian motion.
In the 90's the main purpose of the study of RWs $(Y_n)_{n\ge 0}$  perturbed at extrema, as introduced in Section \ref{ext}, was to prove functional limit theorems for these walks. And indeed, for certain values of the parameters $p$ and $q$ one has 
\[\frac{Y_{[n\cdot]}}{\sqrt n} \overset{J_1}{\Rightarrow }X_{\al,\beta}(\cdot) \ \text{as
}n\to\infty, \]
where $\al=2-1/p$ and $\beta=2-1/(1-q)$, 
see e.g.\ \cite[Th.\ 1.2]{Da96} for an accessible statement and also \cite{D99}.
 
For the ERWs from Sections \ref{1.3} and \ref{1.4} there are the following results:  
\begin{theorem}{\bf ($d=1$, convergence to Brownian motion perturbed at extrema \cite[Th.\ 1, Rem.\  2]{Do11}, \cite[Th.\ 1]{DK})}
\label{ThERWRecLim}
Let $\delta\in[0,1)$.  Assume either 
\begin{equation}\label{DD}
  \text{{\rm (SE), (POS$_1$), (UEL)}, and}\quad \EE\left[\sup\{i\in\N:\ 
    |\omega(0,1,i)-1/2|\ge \epsilon\} \right]<\infty
\end{equation}
or {\rm (IID), (BD),} and {\rm (WEL).}
Then under $P_0,$ \[\dfrac{X_{[n\cdot]}}{\sqrt
  n}\overset{J_1}{\Rightarrow }X_{\delta,-\delta}(\cdot)\ \text{as
}n\to\infty. \]
\end{theorem}
\begin{remark}{\bf (Strips)} 
  {\rm A similar result holds for
    recurrent  ERWs on strips
    $\Z\times\{0,1,\dots,L-1\}$, $L\in\N$, under assumptions similar to 
    (\ref{DD}), see \cite[Th.\ 1]{Do11}.}
\end{remark}
\begin{theorem}{\bf ($d=1$, convergence to the running maximum of Brownian motion \cite[Th.\ 2]{DK})}
\label{ThERWCrit}
Assume {\rm (IID), (BD),} and {\rm (WEL)} and
let $\delta=1$ and $S(t):=\max_{s\le t}B(s)$, $t\ge 0$. Then there exists a
constant $b>0$ such that under $P_0,$
\begin{equation}\label{mex}
\frac{X_{[n\cdot]}}{b \sqrt{n}\log n}\overset{J_1}{\Rightarrow} S(\cdot)\
\text{as }n\to\infty.
\end{equation}
\end{theorem}
\begin{problem}
  {\rm Replace the assumptions {\rm (IID), (BD),} and {\rm (WEL)} of Theorem \ref{ThERWCrit} by (\ref{DD}).  Does the result of
  Theorem~\ref{ThERWCrit} remain true? If yes, can one also remove
(POS$_1$)?}
\end{problem}
Observe that in Theorem~\ref{ThERWCrit} the limiting process is
transient while the original process is recurrent. This might seem
surprising.  For the proof of Theorem~\ref{ThERWCrit} one first shows
that (\ref{mex}) holds with $S_{[n\cdot]}:=\max_{\, 0\le i\le
  [n\cdot]}X_i$ instead of $X_{[n\cdot]}$.  Then it is argued
(\cite[Lem.\ 5.2]{DK}) that with probability approaching 1 the maximum
amount of ``backtracking'' of $X_j$ from $S_j$ for $j\le [Tn]$ is
negligible on the scale $\sqrt{n}\log n$ for every $T<\infty$.

We turn now to the transient case $\delta>1$. 
To describe limit laws of transient
ERWs we need the following notation.  For $\alpha\in (0,2]$ and $b>0$
denote by $Z_{\alpha,b}$ a random variable such that for all $u\in\R$,
  \begin{equation}
    \label{sl}
    \log E\left[e^{iuZ_{\alpha,b}}\right]=
    \begin{cases}
      -b|u|^\alpha(1-i\tan (\pi\alpha/2)\operatorname{sign} u),
      &\text{if } \alpha\ne
      1;\\[2mm]
      -b|u|\left(1+\dfrac{2i}{\pi}\log|u|\operatorname{sign}u\right),&\text{if }
      \alpha=1,
    \end{cases}
  \end{equation}
see e.g.\ \cite[(3.7.11)]{Du}.
  Here we agree to set $\operatorname{sign} 0$ to $0$. Observe also
  that $Z_{2,b}$ is just a centered normal random variable with
  variance $2b$. 
\begin{theorem}\label{llth}{\bf ($d=1$, convergence of one-dimensional distributions) 
} Assume {\rm (IID), (BD),} and {\rm (WEL)} and let 
$v$ be the speed of the ERW as in Theorem \ref{lln}. The following results hold under $P_0$.
  \begin{itemize}
  \item [(i)] If $\delta\in(1,2)$ then there is a constant $b>0$ such
    that as $n\to\infty$,
    \begin{equation*}
      \frac{T_n}{n^{2/\delta}}\Rightarrow
      Z_{\delta/2,b}\label{i}\quad\text{and}
      \quad\frac{X_n}{n^{\delta/2}}\Rightarrow
      (Z_{\delta/2,b})^{-\delta/2}.
    \end{equation*}
  \item [(ii)] If $\delta=2$ then there is a positive constant $c$
    such that
    \begin{equation}\label{ii}
      \frac{T_n}{n\log n}\to \frac1{c}\quad\text{and}\quad
      \frac{X_n}{n/\log n}\to
      c\ \ \text{in probability as }n\to\infty.
    \end{equation}
Moreover, there are a positive constant $b$ and
      functions $D(n)\sim \log n$ and $\Gamma(n)\sim 1/\log n$ such
      that as $n\to\infty$,
    \begin{equation}
      \frac{T_n-c^{-1}nD(n)}{n}\Rightarrow
      Z_{1,b}\label{iitx}\quad\text{and}\quad
      \frac{X_n-cn\Gamma(n)}{c^2n\log^{-2} n}\Rightarrow
      -Z_{1,b}.
    \end{equation}
\item [(iii)] If $\delta\in(2,4)$ then as $n\to\infty$,
    \begin{equation*}
       \frac{T_n-v^{-1}n}{n^{2/\delta}}\Rightarrow Z_{\delta/2,b}
\label{iii}\quad\text{and}
\quad\frac{X_n-vn}{v^{1+2/\delta}n^{2/\delta}}\Rightarrow
-Z_{\delta/2,b}.
    \end{equation*} 
  \item [(iv)] If $\delta=4$ then there is a constant $b>0$ such that
    as $n\to\infty$,
    \begin{equation*}
       \frac{T_n-v^{-1}n}{\sqrt{n\log n}}\Rightarrow Z_{2,b}
      \label{iv}\quad\text{and}\quad
      \frac{X_n-vn}{v^{3/2}\sqrt{n\log n}}\Rightarrow -Z_{2,b}.
    \end{equation*}
  \item [(v)] If $\delta>4$ then there is a constant $b>0$ such that
    as $n\to\infty$,
    \begin{equation*}
       \frac{T_n-v^{-1}n}{\sqrt{n}}\Rightarrow Z_{2,b}
      \label{v}\quad\text{and}
      \quad\frac{X_n-vn}{v^{3/2}\sqrt{n}}\Rightarrow -Z_{2,b}.
    \end{equation*}
    \end{itemize}
    Moreover, everywhere above $X_n$ can be replaced by $\inf_{i\ge n}X_i$
    or $\sup_{i\le n} X_i$.
  \end{theorem} 
  \begin{proof} 
    The regeneration structure described in Lemma \ref{rs} and the
    estimates (\ref{tailx}) and (\ref{tailt}) play an important role
    in the proof.  Statements (i) and (\ref{ii}) were proven under
    additional assumptions, in particular (POS$_1$), in \cite[Th.\
    1.1]{BS08b}. These assumptions were removed in \cite[Rem.\
    9.2]{KM11}. Parts (iii) and (iv) were obtained in \cite[Th.\
    1.3]{KM11}. The second claim of part (v) is contained in
    \cite[Th.\ 3]{KZ08}. Although part (v) was not stated or proven in
    \cite{KM11}, its proof is almost the same as the one of part (iii)
    (see \cite[Proof of Th.\ 1.3, pp.\ 593--594]{KM11}). The only
    difference is that one can use standard limit theorems for square
    integrable random variables instead of those quoted in the
    original proof.

    The fluctuation results (\ref{iitx}) do not seem to appear
    anywhere in the ERW literature.  The proof of (\ref{iitx}) can be
    written along the lines of \cite[pp.\,166--168]{KKS75} taking into
    account \cite[Th.\ 2.1, 2.2, Lem.\ 9.1]{KM11}. The fact that
    $X_{\tau_n}-X_{\tau_{n-1}}$ has infinite second moment (see
    (\ref{tailx})) requires an application of a limit theorem from
    renewal theory, which is different from the one used in
    \cite[p.\,166, line 8 from the bottom]{KKS75}.  Namely, the
    deviations of the number of renewals up to time $n$ have
    fluctuations of order $\sqrt{n\log n}$ instead of $\sqrt{n}$
    (combine \cite[Th.\ 6.3.1]{Wh02} with \cite[Ch.\ 9, Sec.\ 6, Th.\
    2]{GS69}). Apart from this, the proof of
    \cite[pp.\,166--168]{KKS75} goes through essentially word for
    word.
   \end{proof}

   It is natural to expect that the same embedded regeneration
   structure, which was useful in the proof of Theorem \ref{llth},
   will allow us to obtain the corresponding functional limit
   theorems. Theorem \ref{llth} suggests that the limit processes
   should be stable processes, which need not be continuous, and this
   brings us to the question of choosing an appropriate topology on
   the Skorokhod space $D[0,\infty)$.

   In 1956, A.\,V.\ Skorokhod introduced 4 metrics on $D[0,1]$
   (\cite{Sk56}). These metrics generate the corresponding topologies
   on $D[0,1]$, which are called $J_1,\ M_1,\ J_2$, and $M_2$. One of
   them, namely $J_1$, is often called ``the (standard) Skorokhod
   topology'' and is probably the most widely used in the study of
   weak convergence of stochastic processes on the Skorokhod space.
   Nevertheless, when the candidate for the limiting process is not
   continuous, one should not necessarily expect the convergence in
   $J_1$ but only in one of the ``less demanding'' topologies, in our
   case $M_1$. This fact is well-known, for example, in the study of
   heavy traffic limits in queuing theory (see \cite{Wh02} and
   references therein).

Let $\Lambda$ denote the set of
continuous strictly increasing mappings of $[0,1]$ onto itself. Recall (\cite[Sec.\ 12]{Bil99}, \cite[Section 3.3 (3.2)]{Wh02})
that for $x_1,x_2\in D[0,1]$ the Skorokhod $J_1$-distance between
$x_1$ and $x_2$  is defined by
\begin{equation}\label{dj}
\rho_{J_1}(x_1,x_2):=\inf_{\lambda\in\Lambda}\left(\sup_{t\in[0,1]}
|x_1(t)-x_2(\lambda(t))|\vee\sup_{t\in
  [0,1]}|\lambda(t)-t|\right).
\end{equation}

  The Skorokhod $M_1$-distance $\rho_{M_1}(x_1,x_2)$ is informally
  defined as ``the distance between the completed graphs of $x_1$ and
  $x_2$, $\Gamma_{x_1}$ and $\Gamma_{x_2}$''. More precisely
  (\cite[Sec.\ 3.3]{Wh02}), for
  $x\in D[0,1]$ let \[\Gamma_x:=\{(z,t)\in\mathbb{R}\times[0,1]:\
  z=\alpha x(t-)+(1-\alpha)x(t)\ \text{for some }\alpha\in[0,1]\}\subset\R^2,
  \] where $x(0-):=x(0)$, and define the order on $\Gamma_x$ as
  follows: $(z_1,t_1)\le (z_2,t_2)$ if either $t_1<t_2$ or $t_1=t_2$
  and $|z_1-x(t_1-)|\le |z_2-x(t_2-)|$. Denote by $\Pi_x$ the set of all
  parametric representations of $\Gamma_x$, i.e.\ continuous
  non-decreasing (with respect to the above defined order on
  $\Gamma_x$) functions $(u,r)$ taking $[0,1]$ onto $\Gamma_x$. For
  $x_1,x_2\in
  D[0,1]$  let \[\rho_{M_1}(x_1,x_2):=\inf_{(u_i,r_i)\in\Pi_{x_i},\, i=1,2}
  \left(\sup_{t\in[0,1]}|u_1(t)-u_2(t)|\vee
  \sup_{t\in[0,1]}|r_1(t)-r_2(t)|\right).\] 
Clearly, $J_1$-convergence
  implies $M_1$-convergence. The converse is not true.  For both
  examples on Figure~\ref{mj}\footnote{Figure~\ref{mj}:
    $x=\dfrac{1}{4}\,\1_{[0,1/2)}+\dfrac{3}{4}\,\1_{[1/2,1]}$. Left:
    $x_n=\dfrac14\,\1_{[0,1/2-1/n)}+\dfrac12\,\1_{[1/2-2/n,1/2-1/n)}+
    \dfrac34\,\1_{[1/2-1/n,1]}$.  Right: $x_n(t)=
  \begin{cases}
     1/4-1/n+2t/(n-2),&\text{ if }0\le t<1/2-1/n;\\
     3/4+n(t-1/2)/2,&\text{ if }1/2-1/n\le t<1/2;\\
     3/4+2(t-1/2)/n,&\text{ if }1/2\le t\le 1.
   \end{cases}$} 
we have $\rho_{M_1}(x_n,x)\to 0$ but
  $\rho_{J_1}(x_n,x)\not\to 0$ as $n\to\infty$. 
\begin{figure}[h]
  \centering
\includegraphics[height=7cm]{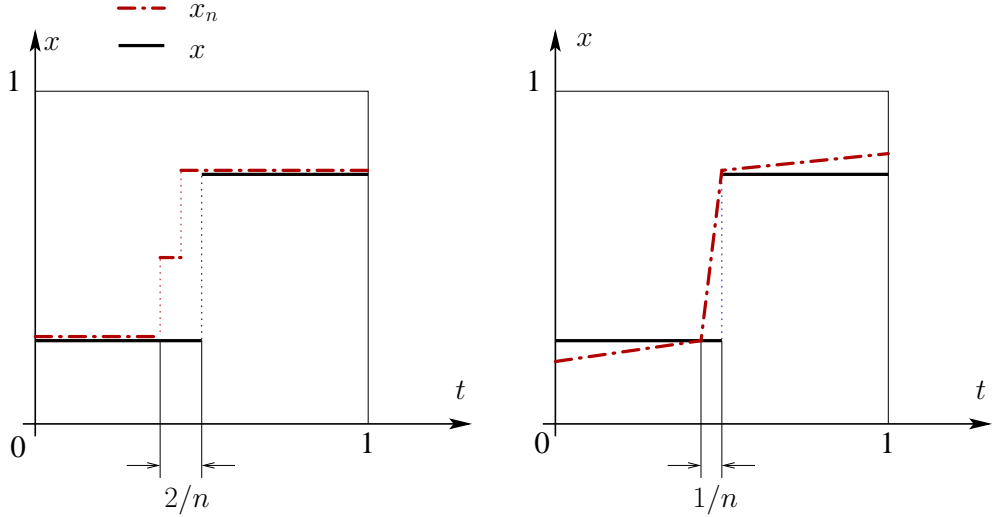}
\caption[]{Unmatched jump of the limiting process.}
\label{mj}
\end{figure}
In the left picture, a
  single jump of the limiting function is achieved by a double
  jump in the approximating functions. The right picture illustrates
  the creation of an unmatched jump in the limiting function by
  continuous approximating functions. Neither situation is an obstacle
  for the $M_1$-convergence but both prevent $J_1$-convergence.

Note however (\cite[Sec.\ 12.4]{Wh02}) that 
\begin{equation}\label{imp}
\mbox{if
$\rho_{M_1}(x_n,x)\to 0$ and $x\in C[0,1]$ then $\rho_{J_1}(x_n,x)\to
0$.}
\end{equation}

  Under natural assumptions (see \cite[Sec.\  13.6]{Wh02}) the inverse map
  is continuous with respect to $M_1$. Consider the (right continuous)
  inverses of the functions in Figure~\ref{mj} (see Figure~\ref{inv})
  and notice that in the left picture of Figure~\ref{inv} we have
  $\rho_{J_1}(x_n^{-1},x^{-1})\to 0$ (and also uniform convergence) and in the right picture only $\rho_{M_1}(x_n^{-1},x^{-1})\to 0$
  as $n\to\infty$ (but neither uniform nor $J_1$-convergence). 
\begin{figure}[h]
  \centering
\includegraphics[height=6.3cm]{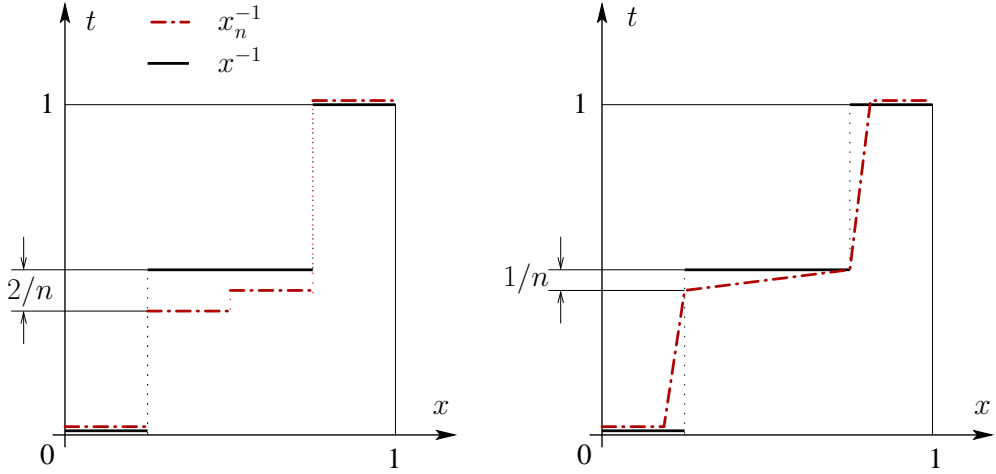}
\caption{The convergence of inverse mappings.}
\label{inv}
\end{figure}
This example
  illustrates that ``strengthening'' of convergence might or might
  not occur when taking the inverse (compare the statements of
  Theorems~\ref{1to2} and \ref{2up} for $\delta\in(2,4)$ below).

We defined the $M_1$-distance on $D[0,1]$. This definition extends
to $D[0,T]$, $T>0$, in an obvious way. For $x_1,x_2\in D[0,\infty)$ we let
$\rho_{M_1}^T(x_1,x_2)$ denote the $M_1$-distance between the
restrictions of $x_1$ and $x_2$ to the space $D[0,T]$. Then (see \cite[Sec.\
3.3, (3.6)]{Wh02})
\begin{equation*}
 \rho_{M_1}^\infty(x_1,x_2):= \int_0^\infty
e^{-T}(\rho_{M_1}^T(x_1,x_2)\wedge 1)\,dT. 
\end{equation*}

The distance $\rho_{J_1}^\infty(x_1,x_2)$ can be defined by
replacing $M_1$ with $J_1$ everywhere in the preceding paragraph. 
For further information about the space $D[0,\infty)$ with metric
$\rho^\infty_{M_1}$ the reader is referred to \cite[Sec.\ 12.9]{Wh02}.

We shall return to ERWs. Using the notation introduced in
Theorem~\ref{llth} define processes
$\eta_n=(\eta_n(t))_{t\ge 0}$ and $\xi_n=(\xi_n(t))_{t\ge 0}$ 
as follows.
\begin{alignat}{4}
&\mbox{If $\delta\in(1,2)$ then}\quad&
\eta_n(t)&:=\dfrac{T_{[nt]}}{n^{2/\delta}},\quad&
\xi_n(t)&:=\dfrac{X_{[nt]}}{n^{\delta/2}}.\tag{i}\displaybreak[2]\\
&\mbox{If $\delta=2$
  then}\quad& 
\eta_n(t)&:=\dfrac{T_{[nt]}-c^{-1}[nt]D(n)}{n},\quad&
\xi_n(t)&:=\dfrac{X_{[nt]}-c[nt]\Gamma(n)}{c^2n\log^{-2} n}.\tag{ii}\displaybreak[2]\\
&\mbox{If $\delta\in(2,4)$ then}\quad& 
 \eta_n(t)&:=\dfrac{T_{[nt]}-[nt]/v}{n^{2/\delta}},\quad&
\xi_n(t)&:=\dfrac{X_{[nt]}-[nt]v}{v^{1+2/\delta}n^{2/\delta}}.\tag{iii}\displaybreak[2]\\
&\mbox{If $\delta=4$
  then}\quad& 
\eta_n(t)&:=\dfrac{T_{[nt]}-[nt]/v}{\sqrt{n\log n}},\quad&
\xi_n(t)&:=\dfrac{X_{[nt]}-[nt]v}{v^{3/2}\sqrt{n\log n}}.\tag{iv}\displaybreak[2]\\
&\mbox{If $\delta>4$
  then}\quad& 
\eta_n(t)&:=\dfrac{T_{[nt]}-[nt]/v}{\sqrt{n}},\quad&
\xi_n(t)&:=\dfrac{X_{[nt]}-[nt]v}{v^{3/2}\sqrt{n}}.\tag{v}
\end{alignat}
Note that $\eta_n(1)$ and $\xi_n(1)$ are the random variables considered in Theorem \ref{llth}.
  Recall the random variables
$Z_{\delta/2,b}$, $b>0$, which were defined in (\ref{sl}). 
\begin{theorem}\label{1to2} {\bf ($d=1$, functional limit theorem)} Assume {\rm (IID), (BD),} and {\rm (WEL)}
and let $\delta\in(1,2)$. Then $\xi_n\overset{J_1}{\Rightarrow} \xi$ as
  $n\to\infty$ under $P_0$.  Here $\xi$ is the inverse of a stable
  subordinator $\eta$ for which $\eta(1)\overset{d}{=}Z_{\delta/2,b}$,
  and $b$ is the same as in part (i) of Theorem~\ref{llth}.
\end{theorem}
We denote by $\overset{M_1}{\Rightarrow} $ the weak convergence on
$D[0,\infty)$ in $M_1$.
\begin{theorem}
  \label{2up} {\bf ($d=1$, functional limit theorem)} Assume {\rm (IID), (BD),} and {\rm (WEL)}.
If $\delta\in(2,4)$ then
  $\xi_n\overset{M_1}{\Rightarrow} \xi$ as $n\to\infty$ under $P_0$. Here $\xi$
  is a stable process with index $\delta/2$ such that
  $\xi(1)\overset{d}{=} -Z_{\delta/2,b}$. 
  If $\delta\ge 4$ then $\xi_n\overset{J_1}{\Rightarrow} \sqrt{2b}\,B
  \ \text{ as }n\to\infty$. The constant $b$ is the same as in the
  corresponding part of Theorem~\ref{llth}.
\end{theorem}

\begin{remark}
{\rm {\bf (case $\delta=2$)}
  The case $\delta=2$ should have been included in
  Theorem~\ref{2up}. However, the results of \cite{Li78} do not
  cover this case, and a separate proof is needed.
}
\end{remark}

\begin{remark} {\rm {\bf (no $J_1$-convergence)} In the first part of Theorem~\ref{2up}
    ($\delta\in(2,4)$) the convergence in $M_1$ can not be replaced
    with convergence in $J_1$. The intuitive explanation is that
    the jumps of the limiting process $\xi$ are unmatched by those of
    the approximating processes $\xi_n$ since the jumps of $\xi_n$
    are bounded by $2/n^{2/\delta}$. This phenomenon was illustrated
    by examples in Figure~\ref{mj} except that the jumps in the
    limiting process $\xi$ are due to centering and, therefore, are
    negative. See also \cite[Section 3.3 and Chapter 6]{Wh02}. 
  }
\end{remark}

\begin{remark}
{\rm {\bf (convergence of $\eta_n$)}
 The weak convergence in $M_1$ under $P_0$ of the processes $\eta_n$
  to the stable process $\eta$ which is characterized by the
  corresponding one-dimensional distribution of Theorem~\ref{llth} can
  also be established (see \cite[Th.\ 2(i) with $g(t)\equiv 1,\ q=1,
  f(t)=h(t)=t$]{Li78} for $\delta>2$ and Remark~\ref{t1to2} below for
  $\delta\in(1,2)$), but, since our main interest is in $\xi_n$, we
  omit the corresponding results.  }
\end{remark}

For the proofs of Theorems \ref{1to2} and \ref{2up} we shall need the following result from \cite{Li78}.

\begin{theorem}\label{lin}
{\bf (\cite[Th.\ 1, Rem.\  2]{Li78})} 
Let $(\zeta_i)_{i\in\N}$ be i.i.d.\ random variables such that
$P[\zeta_1\le 0]=0$, and set $s_0:=0,\ s_n:=\sum_{i=1}^n\zeta_i$ for
$n\in\N$.  Suppose that $\mu:=E[\zeta_1]<\infty$ and the
$D[0,\infty)$-valued process $w=(w(t))_{t\ge 0}$ satisfies: $w(0)=0$
and
\begin{equation}\label{C1}
\begin{array}{l}
  (\zeta_n,w(s_{n-1}+\cdot\zeta_n)-w(s_{n-1})), \ n\ge 1,\ \text{are i.i.d.}\\
  \text{random variables with values in }[0,\infty)\times D[0,1],
\end{array}
\end{equation}
and for some norming constants $a(n)$, $n\in
\N$, \[\dfrac{w(s_{[n\cdot]}) }{a(n)}\overset{M_1}{\Rightarrow}
Y(\cdot)\qquad\mbox{as $n\to\infty$},\] where $Y$ is a stable process
with index $\alpha$.  Moreover, define
\begin{equation*}
 M_n^{(2)}(w):=\Bigg(\sup\limits_{0=t_0<t_1<t_2<t_3=1}\sum_{i=1}^3|w(s_{n-1}+t_i\zeta_n)-w(s_{n-1}+t_{i-1}\zeta_n)|\Bigg)-
  |w(s_n)-w(s_{n-1})|.
\end{equation*}
Then \[ \dfrac{w(n\cdot)}{\mu^{-1/\alpha}a(n)}\overset{M_1}{\Rightarrow}
Y(\cdot)
\quad\text{if and only if}\quad
\dfrac{M_{[n\cdot]}^{(2)}(w)}{a(n)}\overset{M_1}{\Rightarrow}0.\]
\end{theorem}
\begin{proof}[Proof of Theorem~\ref{1to2}]
  By (\ref{r2}) of Theorem \ref{rt} the ERW is transient to the right.
  Therefore, the regeneration structure described in Lemma \ref{rs}
  exists and we may set $\zeta_i=X_{\tau_{i+1}}-X_{\tau_{i}}$ for
  $i\in\N$ and $w(t)=T_{[t]+X_{\tau_1}}-\tau_1$ for $t\in [0,\infty)$
  and have that $(\zeta_i)_{i\ge 1}$ is i.i.d.\ and
  $\mu=E_0[\zeta_i]<\infty$. Moreover, (\ref{C1}) is
  satisfied. Therefore, by (\ref{tailx}) and \cite[Th.\
  4.5.3]{Wh02}, \[\frac{w(s_{[n\cdot]})
  }{n^{2/\delta}}=\frac{\tau_{[n\cdot]+1}-\tau_1}{n^{2/\delta}}\overset{J_1}
  {\Rightarrow} Y(\cdot),\] where $Y$ is a stable subordinator.
  Moreover, $M_{[nt]}^{(2)}(w)\equiv 0$, since $w(t)$ is monotone. By
  Theorem~\ref{lin} we conclude that
  $(n/\mu)^{-2/\delta}(T_{[n\cdot]+X_{\tau_1}}-\tau_1)\overset{M_1}{\Rightarrow}
  Y(\cdot)$ as $n\to\infty$. Since $\tau_1$ is $P_0$-a.s.\ finite, the
  convergence-together theorem \cite[Th.\ 11.4.7]{Wh02} implies that
  \begin{equation}\label{etm}
    \frac{T_{[n\cdot]+X_{\tau_1}}}{(n/\mu)^{2/\delta}}
    \overset{M_1}{\Rightarrow}Y(\cdot).
  \end{equation}
  Taking the inverses (see \cite[Cor.\ 13.6.3]{Wh02}), we get
  $n^{-\delta/2}(S_{[n\cdot]}-X_{\tau_1})\overset{M_1}{\Rightarrow}\mu
  Y^{-1}(\cdot)$, where $S_n=\max_{0\le j\le n}X_j$. Since $\tau_1$ is
  $P_0$-a.s.\ finite, we conclude that
  $n^{-\delta/2}S_{[n\cdot]}\overset{M_1}{\Rightarrow}\mu
  Y^{-1}(\cdot)$. Next we note that $Y^{-1}$ has continuous
  trajectories (\cite[Lem.\ III.17]{Ber96}). This allows us to strengthen the
  convergence to $J_1$ (see (\ref{imp})), i.e.\
  $n^{-\delta/2}S_{[n\cdot]}\overset{J_1}{\Rightarrow}\mu
  Y^{-1}$. Part (i) of Theorem~\ref{llth} identifies the
  one-dimensional distributions of the limiting processes.

  Finally, we would like to replace the maximal process $S$ with the
  original walk $X$. It is enough to consider restrictions of both
  processes to intervals $[0,T]$, $T>0$, (\cite[Th.\ 16.7]{Bil99}). To
  simplify notation, we consider without loss of generality the time
  interval $[0,1]$ instead of $[0,T]$. The tail estimate
  (\ref{tailt}) for the variables of the i.i.d.\ sequence
  $(\tau_{i+1}-\tau_i),\ i\ge 1$, implies that for every $\nu>0$ there
  is $K>0$ such that $P_0[\tau_{[Kn^{\delta/2}]}\le n]<\nu$ for all
  $n\in\N$. For every $\epsilon>0$ we have
  \begin{align*}
\lefteqn{P_0\left[\max_{0\le t \le 1}(S_{[nt]}-X_{[nt]})>\epsilon
      n^{\delta/2}\right]\le \nu+ P_0\left[\max_{0\le m \le
        n}(S_{m}-X_{m})>\epsilon n^{\delta/2},\
      \tau_{[Kn^{\delta/2}]}> n\right]}\hspace*{30mm}\displaybreak[3]\\
&\le\nu+P_0\left[\max_{1\le i\le
        [Kn^{\delta/2}]}(X_{\tau_i}-X_{\tau_{i-1}})>\epsilon
      n^{\delta/2}\right]\displaybreak[3]\\
&\le \nu+P_0[X_{\tau_1}>\epsilon
    n^{\delta/2}]+P_0\left[\max_{2\le i\le
        [Kn^{\delta/2}]}(X_{\tau_i}-X_{\tau_{i-1}})>\epsilon
      n^{\delta/2}\right]\displaybreak[3]\\
&\stackrel{(\ref{tailx})}{\le} 2\nu+1-\left(1-\frac{2C_1}{(\epsilon
        n^{\delta/2})^\delta}\right)^{Kn^{\delta/2}}<3\nu
  \end{align*}
  for all sufficiently large $n$. Now apply the convergence-together theorem.
\end{proof}

\begin{remark}\label{t1to2}{\rm {\bf (convergence of $\eta_n$ for
      $\delta\in(1,2)$)} From (\ref{etm}) it is only one step from the
    conclusion that $\eta_n\overset{M_1}{\Rightarrow}\eta$, where
    $\eta$ is a stable subordinator with
    $\eta(1)\overset{d}{=}Z_{\delta/2,b}$. Indeed, a straightforward
    computation shows that for every $T>0$ the $J_1$- (and, therefore,
    $M_1$-) distance between $n^{-2/\delta}T_{[n\cdot]+X_{\tau_1}}$
    and $n^{-2/\delta}T_{[n\cdot]}$ on $D[0,T]$ converges to $0$ in
    $P_0$-probability as $n\to\infty$. The claim follows by the
    convergence-together theorem and part (i) of Theorem~\ref{llth}. }
\end{remark}
\begin{proof}[Proof of Theorem~\ref{2up}] The statement of
  Theorem~\ref{2up} for $\delta>4$ is contained in \cite[Th.\
  3]{KZ08}. Thus, we assume that $\delta\in(2,4]$ (even though the
  proof works for all $\delta>2$). Set
  $\zeta_i=\tau_{i+1}-\tau_i$, $i\ge 1$, and
  $w(t)=X_{[t]+\tau_1}-v[t]-X_{\tau_1}$, $t\ge 0$. By Lemma \ref{rs}
  $(\zeta_i)_{i\ge 1}$ is i.i.d.\ and (\ref{C1}) holds. Moreover, $\mu=E_0[\zeta_i]<\infty$ by
  (\ref{tailt}). 

    By Theorem \ref{lln2}, $E_0[X_{\tau_2}-X_{\tau_1}]=\mu v$. Let $a(n):=n^{2/\delta}$
  if $\delta\in(2,4)$ and $a(n):=\sqrt{n\log n}$ if $\delta=4$. We shall show
    the weak convergence in $J_1$ of 
\begin{equation}\label{mag}
\frac{w(s_{[n\cdot]})
    }{a(n)}=\frac{X_{\tau_{[n\cdot]+1}}-v(\tau_{[n\cdot]+1}-\tau_1)-X_{\tau_1}}
    {a(n)}.
\end{equation}
 Indeed, by (\ref{tailx}), (\ref{tailt}), \cite[Th.\
    4.5.3]{Wh02}, and \cite[Ch.\ 9, Sec.\ 6, Th.\ 2]{GS69}
    \begin{equation*}
      \frac{X_{\tau_{[n\cdot]+1}}-X_{\tau_1}-\mu v[n\cdot]}
      {a(n)}\overset{J_1}{\Rightarrow}0\quad \text{and}\quad 
      \frac{\tau_{[n\cdot]+1}-\tau_1-\mu[n\cdot]}
      {a(n)}\overset{J_1}{\Rightarrow}\tilde{Y}(\cdot),
    \end{equation*}
    where $\tilde{Y}$ is a stable process with index $\delta/2$. The
    convergence-together theorem implies
    that the process in (\ref{mag}) converges in $J_1$ to $Y:=-v\tilde{Y}$.
    Finally, we need to check the condition
    $(a(n))^{-1}M^{(2)}_{[n\cdot]}(w)\overset{M_1}{\Rightarrow}0$. 
    The
    following two simple lemmas accomplish the task. (We state these
    lemmas for all $\delta>2$ setting
    $a(n)=\sqrt{n}$ when $\delta>4$.)
    \begin{lemma}\label{M} Let $\delta>2$. Then, under the conditions stated above,
      \[M^{(2)}_n(w)\le 4(X_{\tau_{n+1}}-X_{\tau_n})\quad\text{for all $n\ge 1$}.\]
    \end{lemma}
    \begin{proof}
 Let $w(t)=w_1(t)-w_2(t)$, where
      $w_1(t):=X_{[t]+\tau_1}$ and $w_2(t):=v[t]+X_{\tau_1}$. Notice
      that the latter process is increasing. Then for $n\in\N$ and
      $0=t_0<t_1<t_2<t_3=1$,
    \begin{eqnarray*}
      \lefteqn{\left(\sum_{i=1}^3|w(s_{n-1}+t_i\zeta_n)-w(s_{n-1}+t_{i-1}\zeta_n)|\right)-
      |w(s_n)-w(s_{n-1})|}\\
    &\le&
      \left(\sum_{i=1}^3|w_1(s_{n-1}+t_i\zeta_n)-w_1(s_{n-1}+t_{i-1}\zeta_n)|
        +w_2(s_{n-1}+t_i\zeta_n)-w_2(s_{n-1}+t_{i-1}\zeta_n)\right)\\
    &&\mbox{}-
      \left((w_2(s_n)-w_2(s_{n-1})- |w_1(s_n)-w_1(s_{n-1})|\right) \ \le\
      4|w_1(s_n)-w_1(s_{n-1})|.
    \end{eqnarray*}
This implies the desired estimate.
    \end{proof}
    \begin{lemma}
      \label{2} Let $\delta>2$. Then, under the conditions stated above,
 \[\max_{1\le m\le
        n}\frac{X_{\tau_{m+1}}-X_{\tau_m}}{a(n)}\to 0\ \ \text{ as $n\to\infty$ in
        $P_0$-probability.}\]
    \end{lemma}
    \begin{proof}
      For every $\epsilon>0,$
     \begin{eqnarray*}
        P_0\left[\max_{1\le m\le
            n}\frac{X_{\tau_{m+1}}-X_{\tau_m}}{a(n)}>\epsilon\right]&\le&
        1-\left(1-P_0 \left[X_{\tau_2}-X_{\tau_1}>\epsilon
            a(n)\right]\right)^n\\ 
        &\overset{(\ref{tailx})}{\le}&
        1-\left(1-\frac{2C_1}{(\epsilon a(n))^{\delta}}\right)^n\to
        0\ \ \text{as }n\to\infty.
      \end{eqnarray*}
    \end{proof}
    By Theorem~\ref{lin} we obtain
    $(a(n))^{-1}(X_{[n\cdot]+\tau_1}-v[n\cdot]-X_{\tau_1})\overset{M_1}
    {\Rightarrow}
    \mu^{-2/\delta}Y(\cdot)$ as $n\to\infty$. Since
    $\tau_1$ is a.s.\ finite and $X$ moves in unit steps, we conclude
    that 
    \[\frac{X_{[n\cdot]}-v[n\cdot]}{a(n)}\overset{M_1}
    {\Rightarrow} \mu^{-2/\delta}Y(\cdot).\] Parts (iii) and (iv) of
    Theorem~\ref{llth} identify the one-dimensional distribution of
    the limiting stable process. If $\delta=4$ the process $Y$ is a
    constant multiple of the standard Brownian motion and, therefore,
    is continuous. By (6.6), for this case we can replace
    $M_1$- with $J_1$-convergence. This completes the proof of
    Theorem~\ref{2up}.
\end{proof}
Finally we quote a functional limit theorem in a different spirit, which can be found in \cite{RS}. Here $\PP$ is not fixed but scaled as well. Roughly speaking, the cookies are being diluted, becoming more and more like placebo cookies while the total drift stored in a cookie stack remains the same.
\begin{theorem}\label{eb}{\bf ($d=1$, excited Brownian motion as limit of ERW \cite[Th.\ 1.4]{RS})}
Assume that $\varphi:\R\to\R$ is Lipschitz continuous and bounded by some finite constant $C$. Define for all $ k\ge C, k\in\N,$ a spatially homogeneous environment $\om_k\in\Om$ by setting for all $z\in\Z, i\in\N,$
\[\om_k(z,1,i)=1-\om_k(z,-1,i)=\frac 12\left(1+\frac 1{2k}\varphi\left(\frac i{2k}\right)\right)\]
and let $(X_{k,n})_{n\ge 0}$ be an ERW in the deterministic environment $\om_k$. 
Furthermore, let $(Y(t))_{t\ge 0}$ be the solution, called {\em excited Brownian motion}, to the stochastic differential equation
\[dY(t)=dB(t)+\varphi(L(t,Y(t)))\ dt,\quad Y_0=0,\]
where 
$L$ is the local time process of $Y$.
Then as $k\to\infty$,
\[\frac{X_{k,[4k^2\cdot]}}{2k}\stackrel{J_1}{\Rightarrow}Y(\cdot) .\]
\end{theorem}

\subsection{Results for $d\ge 1$.
}\label{flt2up}
The law of large numbers considered in Theorem \ref{bal2} 
can be complemented by an averaged central limit theorem.
\begin{theorem}\label{bwl} {\bf ({$d\ge 2$, convergence to Brownian motion \cite[Th.\ 1 (ii)]{BR07}, \cite[Th.\ 1.2 (ii)]{MPRV}})}
Let $\ell\in \R^d\backslash\{0\}$ and assume {\rm (MPRV$_\ell$), (IID)} and {\rm (UEL)}.
Then there exists a non-degenerate
    $d\times d$ matrix $G$ such that with respect
    to $P_0$, 
\[\frac{X_{[n\cdot]}- [n\cdot]
    v}{\sqrt{n}}\overset{J_1}{\Rightarrow} B_G(\cdot)\quad\text{
    as $n\to \infty$},
\]
where $B_G$ is the $d$-dimensional Brownian motion
    with covariance matrix $G$.
 \end{theorem}
A central limit theorem in the original model (BW) was also obtained in \cite[Th.\ 2.2]{vdHH12} by the expansion technique mentioned in Remarks \ref{exp} and \ref{pnc}.
\begin{pe}\label{pee}{\rm ($d\ge 1$, {\bf quenched CLT}) Find measures $\PP$ (not including those which model RWRE)
such that for $\PP$-almost all $\om$ the distribution  under $P_{0,\om}$ of $T_n$ or $X_n$, properly centered and normalized as needed, converges weakly to some non-degenerate distribution as $n\to\infty$. 

(a) Here is a somewhat degenerate example for $d=1$: Assume (IID), $\si^2:=\EE[\om(0,1,1)$
$\om(0,-1,1)]>0$ and let $\PP$-a.s.\ $\om(0,1,i)=1$ for $i\in\{2,3\}$. It is not difficult to see that for $\PP$-almost all $\om$  the random variables $\Delta_k:=T_{k+1}-T_{k},$ $k\ge 0,$ are  independent under $P_{0,\om}$ with $P_{0,\om}[\Delta_k=1]=\om(k,1,1)$ and
$P_{0,\om}[\Delta_k=3]=\om(k,-1,1)$ for $k\ge 1$ and $P_{0,\om}[T_1=2j+1]=\om(-j,1,1)\prod_{k=0}^{j-1}\om(-k,-1,1)$ for $j\in\N_0$. 
By the Lindeberg-Feller central limit theorem (see e.g.\ \cite[Th.\ 3.4.5]{Du} with $X_{n,k}:=(\Delta_k-E_{0,\om}[\Delta_k])/\sqrt{n}$) we have $\PP$-a.s.\ the following convergence in $P_{0,\om}$-distribution:
\[\frac{T_n-E_{0,\om}[T_n]}{2\si\sqrt{n}}\ \Rightarrow\ Z\quad\mbox{as $n\to\infty$, where $Z\sim\mathcal N(0,1)$}.\]
Note that here the centering is random, i.e.\ depends on $\om$.

(b) Consider measures $\PP$ under which not only the stacks of cookies are i.i.d.\ like in (IID) but also the cookies within the stacks, i.e.\ under which the $2d$-valued vectors $\om(z,\cdot,i), z\in\Z^d, i\in\N,$ are i.i.d.. Obviously, under the averaged measure $P_0$ this ERW has the same distribution as a standard RW with i.i.d.\ increments distributed according to $(\EE[(\om(x,e,i)])_{e\in\mathcal E}$ and thus satisfies a central limit theorem, unless in degenerate cases. Nevertheless it is not obvious that 
the distribution of $X$ under $P_{x,\om}$ satisfies for $\PP$-almost all $\om$  a central limit theorem as well. A similar problem  has been intensively studied in the setting of RWs in space-time random environments, see e.g.\ \cite{DL09} and the references therein. The difference between this well-established model and the new model proposed here is that in the new model the transition probabilities would not depend on the absolute time but on the local time.
}
\end{pe}

\section{Further limit theorems}
In \cite{R04} a large deviation principle for a large class of self-interacting RWs on $\Z^d, d\ge 1$, including some ERWs, is proved.  A stronger result for  ERWs in the setting  discussed in Section \ref{1.4}, whose proof uses branching processes, is given in \cite{P}. 
\begin{theorem}{\bf ($d=1$, large deviation principle under the
    averaged measure \cite{R04}, \cite{P})} Assume {\rm (IID), (BD)} and {\rm (WEL)}. Then there is a continuous, convex function $I:[-1,1]\to[0,\infty)$ such that for any Borel set $B\subseteq [-1,1]$,
\[-\inf_{x\in B^\circ}I(x)\leq  \liminf_{n\to\infty}P_0\left[\frac{X_n}{n}\in B\right]\leq  \limsup_{n\to\infty}P_0\left[\frac{X_n}{n}\in B\right]\le -\inf_{x\in \bar B}I(x),\]
where $B^\circ$ denotes the interior and $\bar B$ the closure of $B$. Moreover, if $\delta>2$ and $v$ denotes the speed of the walk then $I(x)=0$ for all $x\in[0,v]$. More precisely, for all $x\in(0,v)$,
\[\lim_{n\to\infty}\frac{\log P_0[X_n/n<x]}{\log n}=1-\frac{\delta}{2}.\]
\end{theorem}
For further properties of the rate function $I$ see \cite[Lem.\ 5.1]{P}.

The maximum local time of transient ERW under the conditions of Section \ref{1.4}
is considered in \cite{RR}. In particular, the following is shown.
\begin{theorem}{\bf ($d=1$, maximum local time \cite[Th.\ 1.2 (ii), Cor.\ 1.4]{RR})}\\
Assume {\rm (IID), (BD)} and {\rm (WEL)} and let $\delta>1$. Denote by
$\xi_n^*:=\max_{x\in\Z}\#\{i\in\{0,\ldots, n\}\mid X_i=n\}$
the largest number of visits to a single site by time $n$. Then $P_0$-a.s.\ for all $\al>1/\delta$, 
\[\lim_{n\to\infty}\frac{\xi_n^*}{n^{1/(2\vee\delta)}\ (\log n)^\al}=0\quad\mbox{and}\quad
\lim_{n\to\infty}\frac{(\log n)^\al\ \xi_n^*}{n^{1/(2\vee\delta)}}=\infty.\]
\end{theorem}
Thus  $(\xi_n^*)_{n\ge 0}$ behaves  in the non-ballistic transient regime  $\delta\in (1,2]$ similarly to its counterpart for the simple symmetric RW whereas in the ballistic regime $\delta>2$ the exact value of $\delta$ plays an important role.
The proof uses branching processes.\vspace*{3mm}

{\bf Acknowledgment:} E. Kosygina would like to thank Thomas Mountford
for discussions.  Her work was partially supported by
the CUNY Research Foundation, PSC-CUNY award \# 64603-00-42, and by
the Simons Foundation, Collaboration Grant in Mathematics 209493.  M.\
Zerner's work was supported by the European Research
Council, Starting Grant 208417-NCIRW.


\vspace*{3mm}

{\sc \small
\begin{tabular}{ll}
Department of Mathematics& \hspace*{20mm}Mathematisches Institut\\
Baruch College, Box B6-230& \hspace*{20mm}Universit\"at T\"ubingen\\
One Bernard Baruch Way&\hspace*{20mm}Auf der Morgenstelle 10\\
New York, NY 10010, USA&\hspace*{20mm}72076 T\"ubingen, Germany\\
{\verb+elena.kosygina@baruch.cuny.edu+}& \hspace*{20mm}{\verb+martin.zerner@uni-tuebingen.de+}
\end{tabular}
}
\end{document}